\documentclass[12pt]{article} 
\usepackage{amsmath,amssymb} 
\usepackage{mathtools}
\usepackage{esint}
\usepackage{graphicx}
\usepackage{color}
\usepackage{amsfonts}
\usepackage{epstopdf}
\usepackage{algorithm}
\usepackage{enumitem}
\setlist[itemize]{leftmargin=0.0mm}
\usepackage[noend]{algpseudocode}
\usepackage{diagbox}
\usepackage{amsthm}
\newtheorem{lemma}{Lemma}

\newtheorem{theorem}{Theorem}[section]
\usepackage{footnote}

\graphicspath{ {./pic/} }
\usepackage{hyperref}
\usepackage{authblk}

\title{Partially explicit splitting scheme with explicit-implicit-null method for nonlinear multiscale flow problems }
\author[1]{Yating Wang } 
\author[2]{Wing Tat Leung} 

\affil[1]{School of Mathematics and Statistics, Xi'an Jiaotong University.}
\affil[2]{Department of Mathematics, City University of Hong Kong.}

\begin{document}

\maketitle
\begin{abstract}
In this work, we present an efficient approach to solve nonlinear high-contrast multiscale diffusion problems. We incorporate the explicit-implicit-null (EIN) method to separate the nonlinear term into a linear term and a damping term, and then utilise the implicit and explicit time marching scheme for the two parts respectively. Due to the multiscale property of the linear part, we further introduce a temporal partially explicit splitting scheme and construct suitable multiscale subspaces to speed up the computation. The approximated solution is splitted into these subspaces associated with different physics. The temporal splitting scheme employs implicit discretization in the subspace with small dimension that representing the high-contrast property and uses explicit discretization for the other subspace. We exploit the stability of the proposed scheme and give the condition for the choice of the linear diffusion coefficient. The convergence of the proposed method is provided. Several numerical tests are performed to show the efficiency and accuracy of the proposed approach.
\end{abstract}

\section{Introduction}
Nonlinear partial differential equations with multiscale properties occurs in many science and engineering application, such as compressible flow in porous media, heat conduction in composite materials, water flows in unsaturated soils, etc. Due to the nonlinearity and disparity between scales of the underlying model, it is computational challenging to solve these equations numerically.

For time-dependent problems with multiscale features, the high-contrast ratios between scales require very fine meshes for spatial discretization, this increases the degrees of freedom significantly and leads to heavy computational burden. Furthermore, the time step size for the temporal discretization in the explicit scheme also relies on the highest magnitude of the multiscale coefficient in order to form a stable solver. Though implicit schemes are unconditionally stable, solving nonlinear problems usually suffer from the iterative methods at every time step. Thus it is important to construct efficient and accurate fully discretized solvers for such problems.

Here, we first briefly review existing spatial and temporal discretization methods for linear multiscale time-dependent problems. Many studies on constructing reduced order models have been proposed to enhance the efficiency of spatial approximation, including local \cite{eh09, aarnesMMS, allaire2005multiscale} and global \cite{POD_notes, benner2015survey, nonlinear_deim} model reduction methods. Local approaches usually create computational models on a coarse grid with desired accuracy, to name a few,  multiscale finite element method (MsFEM) \cite{eh09}, generalized multiscale finite element methods (GMsFEM) \cite{GMsFEM13, MixedGMsFEM}, localized orthogonal decomposition methods (LOD) \cite{maalqvist2014localization}, constraint energy minimizing GMsFEM (CEM-GMsFEM) \cite{cem-gmsfem, chung2019correction}, nonlocal multicontinuum method (NLMC) \cite{NLMC}. For time discretization, many works has been proposed to handle multiscale stiff systems \cite{abdulle2012explicit, giraldo2013implicit, savcenco2007multirate, carciopolo2020conservative, cem-splitting, splitting-multirate}. Among which, the partially explicit temporal splitting method\cite{cem-splitting, splitting-multirate} for linear multiscale flow problem is introduced to produce a splitting scheme with a suitable construction of multiscale subspaces based on CEM-GMsFEM such that the scheme is stable and the time step size is independent of the contrast. 

To handle nonlinear problems, one way to incorporate the advantages of implicit and explicit time marching method is the so-called explicit-implicit-null method (EIN) \cite{EIN, duchemin2022mars, shu_ein_ldg}. It is noticed that the implicit time discretization scheme can be efficient for solving diffusion equations with constant coefficients (without multiscale properties). This observation inspires one to introduce a term with constant diffusion coefficient and then substract that term in the original nonlinear equation, while the the implicit time marching is only used for the constant diffusion term (damping term), and the explicit time marching is adopted for the difference between the nonlinear term and the constant diffusion term. It has been shown that the EIN is stable and convergent with a suitable choice of constant diffusion.

In this work, we would like to speed up the simulation of nonlinear multiscale flow problems with the help of EIN. The main idea is to introduce the EIN method to handle nonlinearity. However, the diffusion term we will add and substract in the equation is no longer a constant anymore, but is the dominant multiscale coefficient in order to capture the underlying property. To deal with this issue, we will bring in the partially explicit splitting scheme to handle the linear multiscale part of the equation. The idea is to construct multiscale spaces based on the dominant multiscale coefficient under the framework of CEM-GMsFEM or NLMC, and split the solution into multiscale subspaces associated with different physics. Then one treats the small number of degrees of freedom corresponding to the subspace representing high-contrast in an implicit manner and handles the another subspace explicitly. The stability and convergence of the proposed scheme can be ensured.

The rest of the paper is organized as follows. In Section \ref{sec:prelim}, we describe the problem setup, the fine discretization of the problem, the EIN scheme and the partially explicit splitting scheme with EIN, as well as the fast approximation method for nonlinear terms. The construction of multiscale spaces is presented in Section \ref{sec:spaces}. Then we study the stability and convergence of the proposed method in Section \ref{sec:stability}. The numerical tests are shown in Section \ref{sec:numerical}. In the end, the conclusion is drawn.

\section{Problem Setup}\label{sec:prelim}
Consider the nonlinear parabolic equation 
\begin{equation}\label{eq:strong_model}
\begin{aligned}
    u_t - \nabla \cdot (\kappa(x,u) \nabla u) &= f \;\;\;\; &&\text{ on } \Omega\times (0,T] \\
    u &= u_D \;\;\;\; &&\text{ on }  \Gamma_D \times (0,T]\\
    \nabla u\cdot \boldsymbol{n} &= 0 \;\;\;\; &&\text{ on }  \Gamma_N \times (0,T]\\
    u &= u_0  \;\;\;\; &&\text{ on } \partial \Omega \times \{0\}
\end{aligned}
\end{equation}
where $\Omega\in \mathbb{R}^d$ is the computational domain, the boundary of the domain $\partial \Omega  = \Gamma_D \cup \Gamma_N$, $\boldsymbol{n}$ is the unit outward normal vector on $\partial \Omega$. We assume the nonlinear diffusion $\kappa(x,u) = \kappa_x(x) \kappa_u(u)$, where $\kappa_x \in L^{\infty}(\Omega)$ is a heterogeneous coefficient with high contrast, that is, the value of the conductivity/permeability in different regions of $\kappa$ can differ in magnitudes, $\kappa_u(u) > 0$ and is bounded and Lipchitz continuous. Denote by the inner product and the norm in $L_2$ by $(\cdot, \cdot)$ and $\|\cdot\|$, respectively.

 
The weak form of the problem is to seek $u(t,\cdot) \in V=\{ u\in H^1(\Omega), u=u_D \text{ on } \Gamma_D\}$ such that
\begin{equation}\label{eq:weak_eq}
    (u_t, v) + a(u;u,v) = (f,v) , \quad \forall v\in V, \quad t\in(0,T]
\end{equation}
with initial condition $u(0)= u_0$ and $a(w;u,v)=\int_{\Omega}\kappa_{x}\kappa_{u}(w)\nabla u\cdot\nabla v$. 

\subsection{Fine and coarse scale approximation}
Let $\mathcal{T}^H$ be a coarse spatial partition of the computational domain $\Omega$, and its uniform refinement $\mathcal{T}^h$ be a fine spatial partition, where $h\ll H$. Let $V_h$ be the linear finite element space with respect to the grid $\mathcal{T}^h$, and $V_H$ be the multiscale space associated with $\mathcal{T}^H$. The construction of the multiscale basis functions will be discussed later.

The full discretization with finite element in space and implicit Euler in time for \eqref{eq:strong_model} is to find $u_h^{n+1}\in V_h$ such that 
\begin{equation}\label{eq:fine_implicit}
   \frac{1}{\Delta t} {(u_h^{n+1}-u_h^{n}, v)}   + (\kappa(x, u_h^{n+1}) \nabla u_h^{n+1}, \nabla v) = (f,v) , \quad \forall v \in V_h
\end{equation}
where $\Delta t$ is the time step size. The above nonlinear problem can be solved using a Newton-Raphson iteration. That is, given $u_h^n$ at each timestep, define the residual in terms of a basis function $\phi_i$ at the $k$-th iteration,
\begin{equation}\label{eq:residual}
    R(u_h^{n,k}, \phi_i) = (u_h^{n,k} - u_h^n,   \phi_i) +\Delta t (\kappa(x, u_h^{n,k}) \nabla u_h^{n,k}, \nabla \phi_i) -\Delta t(f,\phi_i).
\end{equation}
Let $J$ be the Jacobian matrix of the residual $R$, then the Newton-Raphson approach is to solve a linearized problem iteratively,
\begin{equation}\label{eq:NR}
    \sum_i J(u_h^{n,k}, \phi_i) \delta_i =  -R(u_h^{n,k}, \phi_i), \quad \forall \; i
\end{equation}
where $\delta_i$ is the $i$-th coefficient of the update $\delta = \sum_i \delta_i \phi_i$, and $u_h^{n,k+1} = u_h^{n,k} + \delta$. The iterations are stopped at each timestep when the norm of the update is within a small tolerance, and $u_h^{n+1}$ will take the solution at the last Newton-Raphson iteration. 


To improve the computational efficiency for the multiscale problem, a coarse scale approximation is needed. For the nonlinear problem considered in this work, one approach is to construct appropriate multiscale space $V_H$ associated with coarse grid mesh $\mathcal{T}^H$, and utilize the reduced dimensional space for spatial discretization and still adopt the implicit Euler method with newton type iterations for temporal discretization. In this work, the coarse scale space will be computed under the framework of constraint energy minimizing generalized multiscale finite element method (CEM-GMsFEM) and nonlocal multicontinuum upscaling method (NLMC). The details of these methods will be presented in Section \ref{sec:spaces}. With the constructed reduced dimensional space $V_H$, the coarse scale approximation is to find $u_H^{n+1}\in V_H$ such that 
\begin{equation}\label{eq:fine_weak}
   \frac{1}{\Delta t} {((u_H^{n+1}, v)-(u_H^{n}, v))}   + (\kappa(x, u_H^{n+1}) \nabla u_H^{n+1}, \nabla v) = (f,v) , \quad \forall v \in V_H.
\end{equation}
The above nonlinear system can still be solved with the Newton-Raphson approach \cite{poveda-compressible, yang-poddiem}.

\subsection{Explicit-Implicit-Null (EIN) approach}
To avoid the newton iterations in each time step, the Explicit-Implicit-Null method was proposed \cite{EIN, duchemin2022mars, shu_ein_ldg}. The idea is to add and subtract a term with constant diffusion coefficient $\nabla \cdot (\kappa_0 \nabla u)$ in the nonlinear equation, and then apply implicit time marching for the added linear term, and explicit time marching for the rest. To be specific, EIN solves the following 
\begin{equation}\label{eq:fine_ein}
\begin{aligned}
       &\frac{1}{\Delta t} {((u^{n+1}, v)-(u^{n}, v))}  + \underbrace{(\kappa_0 \nabla u^{n+1}, \nabla v)}_{(I)}  \\
       &+ \underbrace{(\kappa(u^{n}) \nabla u^{n}, \nabla v) -(\kappa_0 \nabla u^{n}, \nabla v)}_{(II)}   = (f,v), 
\end{aligned}
\end{equation}
for all test function $v$, and $\kappa_0$ is a chosen constant coefficient such that the stability of the scheme is ensured.

\subsection{Partially explicit splitting scheme with EIN }

In our work, we adopt the EIN idea to handle the non-linear problem \eqref{eq:strong_model} with coefficient $\kappa(x,u)$. The diffusion coefficient $\kappa_0$ in \eqref{eq:fine_ein} was chosen to be a fixed heterogeneous field which dominates the diffusion and does not change with time. However, after the EIN procedure, if we directly apply the implicit time marching on term (I) in \eqref{eq:fine_ein}, the computational cost could still be high due to the high heterogeneity if $\kappa_0$. To avoid this issue, we further introduce a multiscale splitting scheme to improve the efficiency.

We recall that, for the linear time-dependent diffusion problem with high-contrast coefficients under suitable boundary and initial conditions, 
\begin{equation*}
    u_t - \nabla \cdot (\kappa_0 \nabla u) = f,
\end{equation*}
a temporal splitting method combined with appropriate spatial multiscale method \cite{cem-splitting} was proposed recently to produce a contrast-independent partially explicit time discretization scheme. 
The crucial ingredient of this approach lies in the construction of multiscale subspaces which guarantees the stability of implicit-explicit scheme and the time step size independent of the contrast. To be specific, suppose the multiscale space $V_H$ can be decomposed into two subspaces
\begin{equation*}
    V_H = V_{H,1} + V_{H,2}.
\end{equation*}
In \cite{cem-splitting}, the whole space $V_H$ contains basis functions computed using CEM-GMsFEM \cite{cem-gmsfem}. Then the dominant basis functions representing the high-contrast features have very few degrees of freedom and will form the subspace $V_{H,1}$. The additional complementary basis functions will form the subspace $V_{H,2}$ and lead to a contrast-independent condition for the explicit part of the discretization. Then the partially explicit temporal splitting scheme is to find $u_{H,1}^{k} \in V_{H,1}$ and $u_{H,2}^{k} \in V_{H,2}$, 
\begin{equation} \label{eq:partial_exp}
\begin{aligned}
    \left( \frac{u_{H,1}^{k+1} -u_{H,1}^k }{\Delta t}, v_1 \right) + \left( \frac{u_{H,2}^{k} -u_{H,2}^{k-1} }{\Delta t }, v_1 \right) + (\kappa_0 \nabla (u_{H,1}^{k+1} +u_{H,2}^{k}), \nabla v_1) &= (f_1^k, v_1),\\
        \left( \frac{u_{H,2}^{k+1} -u_{H,2}^k }{\Delta t}, v_2 \right) +\left( \frac{u_{H,1}^{k} -u_{H,1}^{k-1} }{\Delta t}, v_2 \right) + (\kappa_0  \nabla(u_{H,1}^{k+1} +u_{H,2}^{k}), \nabla v_2) &= (f_2^k, v_2), 
\end{aligned} 
\end{equation}
for all $k$ and $\forall v_1 \in V_{H,1}, \forall v_2 \in V_{H,2}$. The two sub-equations will be solved sequentially and the solution at time step $n+1$ will be $u_H^{n+1} = u_{H,1}^{n+1} + u_{H,2}^{n+1}$.

To construct multiscale basis functions associated with physics and bring in local adaptivity for accurate and efficient solution approximations, a multirate method was introduced in \cite{splitting-multirate}. In this framework, the two physically meaningful multiscale subspaces are constructed to represent the underlying features of the solution in the corresponding spatial regions, and the time step sizes in two sub-equations can be different in the splitting scheme. 

In this work, we will incorporate EIN with the above-mentioned partially explicit scheme for nonlinear problems.
\begin{equation*}
    u_t - \nabla \cdot (\kappa(x,u) \nabla u) = f,
\end{equation*}
where $\kappa(x,u) = \kappa_x(x)\kappa_u(u)$, $\kappa_{x}\in L^{\infty}$
is a high-contrast media and $\kappa_{u}$ is a positive function. 
Let $\kappa_{0}(\cdot):=\kappa_{x}(\cdot)\kappa_{u}(\tilde{u}(\cdot))$ for a given $\tilde{u}\in H^{1}$.

The idea is to solve the following EIN-splitting scheme:

		\begin{equation} \label{eq:partial_ein_exp1}
		\begin{aligned} 
			& \frac{1}{\Delta t}(u_{H,1}^{k+1} -u_{H,1}^k, v_1) +\frac{1}{\Delta t } (u_{H,2}^{k} -u_{H,2}^{k-1}, v_1)  + (\kappa_0  \nabla (u_{H,1}^{k+1} +u_{H,2}^{k}), \nabla v_1)  \\
			& + (\kappa(x, u_{H,1}^{k} +u_{H,2}^{k}) \nabla (u_{H,1}^{k} +u_{H,2}^{k}), \nabla v_1) - (\kappa_0  \nabla(u_{H,1}^{k} +u_{H,2}^{k}), \nabla v_1) = (f_1^{k},v_1),
		\end{aligned} 
	\end{equation}
    \begin{equation}\label{eq:partial_ein_exp2}
		\begin{aligned}
        &\frac{1}{\Delta t} (u_{H,2}^{k+1} -u_{H,2}^k ,v_2) +\frac{1}{\Delta t}(u_{H,1}^{k} -u_{H,1}^{k-1}, v_2)+ (\kappa_0 \nabla (u_{H,1}^{k} +u_{H,2}^{k}), \nabla v_2) \\
		& + (\kappa(x, u_{H,1}^{k} +u_{H,2}^{k})\nabla (u_{H,1}^{k} +u_{H,2}^{k}),\nabla v_2)- (\kappa_0 \nabla (u_{H,1}^{k+1} +u_{H,2}^{k}),\nabla v_2) = (f_2^{k},v_2),
		\end{aligned} 
    \end{equation}
for all $v_1 \in V_{H,1}$ and $v_2 \in V_{H,2}$.

We remark that, with the help of EIN, we can avoid the iterative procedure for the nonlinear problem. The partially explicit scheme can further reduce the computational cost by splitting the multiscale space $V_{H}$ into two subspaces, and lead to a contrast-independent stability condition for the temporal discretization. 

\subsection{Approximation with proper orthogonal decomposition (POD) and discrete empirical interpolation method (DEIM)}

Note that we have to assemble matrices with respect to the term nonlinear term $(\kappa(x, u)\nabla u, \nabla v)$ appeared in all the above mentioned schemes on-the-fly and it could not be precomputed. This could be time consuming.

In this section, we introduce the DEIM approach to approximate 
nonlinear functions by combining projection with interpolation.

We first briefly introduce the idea of DEIM. Let $f(\mu)$ be a nonlinear function with parameter $\mu$. Suppose $U=[v_1, \cdots, v_m] \in \mathbb{R}^{n\times m}$ is a matrix consisting of a set of $m$ basis which will span a subspace we would like to project $f$ onto. 

To get the projection basis $\{u_1, \cdots, u_m\}$, one can apply the proper orthogonal decomposition (POD) technique on some snapshots $\{f(\mu_1), \cdots f(\mu_N)\}$ which computed from the full order model. The idea of POD is to find an orthonormal set of basis $\{v_1, \cdots v_m\}$ where $m << N$ to minimize the quantity
\begin{equation*}
    \sum_{i=1}^{N} \lVert f(\mu_i) - \sum_{j=1}^{m} (f(\mu_j)^T u_j )u_j\rVert ^2_2.
\end{equation*}

This is achieved by the procedures summarized as follows.

Let $Y = [f(\mu_1), \cdots f(\mu_N)]$, and the SVD of $Y$ is $Y=V\Sigma W^T$ where $\Sigma = diag(\sigma_1, \cdots, \sigma_r)$ and $\sigma_i$ are placed in a descending order, $r$ is the rank of the $Y$. Then the POD basis $\{u_i\}$ are chosen to be the first $m$ column vectors in the left matrix $V$. 
Define the fractional energy as $E = \frac{\sum_{i=1}^m \sigma_i}{\sum_{i=1}^r \sigma_i}$,
then the number of basis $r$ is selected such that the fractional energy reaches a threshold $E=\epsilon_E$, $0.9<\epsilon_E<1$.

With the POD basis $U$, we can approximate $f$ by 
\begin{equation}\label{eq:deim_cmu}
    f(\mu) \approx U c(\mu),
\end{equation}
 where $c(\mu)$ is the coefficient vector. Define a matrix
\[ \mathcal{P}=[e_{P_1}, \cdots, e_{P_m}] \in \mathbb{R}^{n\times m}\]
where $e_{P_i}$ is the $P_i$-th column of the n-th order identity matrix, and $P_1, \cdots, P_m$ are the interpolation indices. The DEIM algorithm finds a set of indices $P_1, \cdots, P_m$ given the basis $U$ via a greedy approach. The first index $P_1$ corresponds to the entry where $u_1$ has the largest maginitude. The rest of the indices $P_i (i\geq 2)$ are chosen such that the $P_i$-th component of the residual $v_i -Wc$ is the largest, where $c$ is computed from $( \mathcal{Q}^T W)c =  \mathcal{Q}^T v_i$, and $W$ is matrix consisting of the first $i-1$ basis in $U$, $Q$ is the matrix consisting of the normal vectors $\{e_{P_1}, \cdots, e_{P_{i-1}}\}$. 

Finally, the vector $c(\mu)$ in \eqref{eq:deim_cmu} can be found from $\mathcal{P}^T f(\mu) = (\mathcal{P}^TU) c(\mu)$,
thus 
\begin{equation*}
\begin{aligned}
c(\mu) = (\mathcal{P}^T U)^{-1}  \mathcal{P}^T f(\mu),\;\;\; f(\mu) \approx  U(\mathcal{P}^T U)^{-1}  \mathcal{P}^T f(\mu)
\end{aligned}
\end{equation*}

In this work, we approximate $\kappa(x,u)$ by the DEIM-POD method. That is, $\kappa(x, u) \approx \sum_{i=k}^m c_k v_k$,
where $c_k$ are coefficients to be determined online (at each time step), $v_k$ are precomputed basis functions.

Let $[A_f(v_k)]_{ij} = \int_{\Omega} v_k \nabla \phi_i \cdot \nabla \phi_j$ be the fine scale stiffness matrix associated with each $v_k (k=1,\cdots,m)$, they can be computed offline. Then the fine scale stiffness matrix at the online stage can be assembled using
\begin{equation*}
   A_f(\kappa(x,u)) =  \sum_{k=1}^m c_k  A_f(v_k)
\end{equation*}
where $\textbf{c} = (\mathcal{P}^T U)^{-1}  \mathcal{P}^T \kappa(u)$.

Now, let $\Psi_1 \in \mathbb{R}^{D\times d_1}$ and $\Psi_2 \in \mathbb{R}^{D\times d_2}$ be the matrices whose columns are the bases of $V_{H,1}$, $V_{H,2}$, respectively. Then one can precompute the folloiwng coarse scale matrices 
\begin{equation*}
\begin{aligned}
 A_{H,1}(v_k) = \Psi_1^T A_f(v_k) \Psi_1, \quad  A_{H,2}(v_k) = \Psi_2^T A_f(v_k) \Psi_2, \quad A_{H,12} = \Psi_1^T A_f(v_k) \Psi_2,
\end{aligned}
\end{equation*}
for all $k=1,\cdots,m$. 

The online coarse scale matrices can be computed at a similar effective manner 
\begin{equation*}
\begin{aligned}
    & A_{H,1}(\kappa(x,u)) =  \sum_{k=1}^m c_k A_{H,1}(v_k), \quad A_{H,2}(\kappa(x,u)) =  \sum_{k=1}^m c_k A_{H,2}(v_k),  \\ 
    & A_{H,12}(\kappa(x,u)) =  \sum_{k=1}^m c_k A_{H,12}(v_k).
\end{aligned}
\end{equation*}

\section{Construction of multiscale spaces}\label{sec:spaces}
In this section, we will present the construction of multiscale spaces. 

Let $\{K_i\} ( i = 1, \cdots, N_c)$ be the set of coarse blocks associated with coarse mesh $\mathcal{T}^H$. Denote by $K_i^+$ the oversampled region for $K_i$, and $K_i^+$ consists of $K_i$ as well as a few layers of coarse blocks neighboring $K_i$. We write $V(K_i)$ as the restriction of $V=H_0^1(\Omega)$ on $K_i$.

The idea of NLMC is to first construct auxiliary spaces $V_{\text{aux},1}$ and $V_{\text{aux},2}$, and then solve some localized energy minimizing problems to get multiscale basis and form the multiscale subspaces $V_{H,1}$ and $V_{H,2}$. The auxiliary spaces in NLMC are simplified, which saves computational efforts. However, to improve approximation accuracy, one may adopt an enriched version of NLMC.

\subsection{Construction of basis based on NLMC} \label{sec:V2}
In many applications, the configuration of $\kappa$ such as the locations of the fractures or channels are known, and the magnitude of $\kappa$ in the background and in the fractures are constants. That is, denote by the computational domain $\Omega = \Omega_m \cup \Omega_f$, where the subscripts $m$ and $f$ denote the matrix and fractures, respectively, and $\kappa|_{\Omega_m} \ll \kappa|_{\Omega_f}$. Under this assumption, one can simplify the construction of the basis functions via the NLMC approach \cite{NLMC, zhao2020analysis}.

To be specific, given a coarse block, denote by $K_i =K_{i,f} \cup K_{i,m}$, where $K_{i,f} :=  \{f_j^{(i)}, \; j = 1, \cdots, m_i\}$ is the high-contrast channelized region, $m_i$ is the number of non-connected fractures in $K_i$, and $K_{i,m}$ is its complement in $K_i$. The two simplified auxiliary spaces are
\begin{equation}\label{eq:nlmc_aux}
\begin{aligned}
V_{\text{aux},1}^{(i)} & = \text{span}\{\phi_{\text{aux},k}^{(i)}\; | \phi_{\text{aux},k}^{(i)} = 0 \text{ in } K_{i,m}, \; \phi_{\text{aux},k}^{(i)} = \delta_{jk} \text{ in } f_j^{(i)},  \;\; k =1, \cdots, m_i\}\\
V_{\text{aux},2}^{(i)} & = \text{span}\{\phi_{\text{aux},0}^{(i)}\; |
\phi_{\text{aux},0}^{(i)} = 1 \text{ in } K_{i,m}, \;
\phi_{\text{aux},0}^{(i)} = 0 \text{ in } K_{i,f}\}
\end{aligned}
\end{equation}

Then, we define $V_{\text{aux},1}^{(i)} :=  \text{span} \{\phi_{\text{aux},1}^{(ik)}, \quad 1 \leq k \leq l_{i,1}\}$, where $1\leq i \leq N_c$ and $N_c$ is the number of coarse blocks. Then the global auxiliary space $V_{\text{aux}} = \bigoplus_{i} V_{\text{aux}}^{(i)}$.




Then the NLMC basis are obtained by finding $\psi_m^{(i)} \in V_0(K_i^+)$ and $\mu_0^{(j)}, \mu_n^{(j)} \in \mathbb{R}$ from 
the following localized constraint energy minimizing problem
\begin{equation}\label{eq:basis}
\begin{aligned}
&a(\psi_m^{(i)}, v) + \sum_{K_j \subset K_i^+} \left(\mu_0^{(j)} \int_{K_{j,m}}v  + \sum_{1 \leq n \leq m_j} \mu_n^{(j)} \int_{ f_n^{(j)}} v \right) = 0, \quad   \forall v\in V_0(K_i^+), \\
&\int_{K_{j,m}}\psi_m^{(i)}    = \delta_{ij} \delta_{m0}, \quad   \forall K_j \subset K_i^+, \\
&\int_{f_n^{(j)}} \psi_m^{(i)}    = \delta_{ij} \delta_{mn}, \quad   \forall f_n^{(j)}, \;  \forall K_j \subset K_i^+.
\end{aligned}
\end{equation}
  The NLMC basis functions are then $ \{ \psi_m^{(i)}, \; 0 \leq m \leq  m_i, 1 \leq i \leq N_c\}$. We remark that the resulting basis separates the matrix and fractures automatically, and have spatial decay property \cite{cem-gmsfem, NLMC, zhao2020analysis}. Furthermore, since the local auxiliary basis are constants within fractures and the matrix, the solution variables obtained from the upscaled equation are physically meaningful, i.e, they denote the solution averages in the fracture and matrix continuum within each coarse region.
  
Finally, we divide the NLMC space into two subspaces 
\begin{equation*}
    \begin{aligned}
        V_{H,1} &= \text{span}\{ \psi_m^{(i)}, \; 1 \leq m \leq  m_i, 1 \leq i \leq N_c\},\\
        V_{H,2} &= \text{span}\{ \psi_0^{(i)}, \; 1 \leq i \leq N_c\}.
    \end{aligned}
\end{equation*}




\subsection{Construction of basis based on enriched NLMC (ENLMC)}

To improve approximation accuracy, we introduce an enriched version of NLMC in this section.
To construct the first auxiliary space $V_{\text{aux},1}$ which can capture the heterogeneity, especially high-constrast channels in the underlying media, one can solve the following spectral problem on $K_{i,f}$ of each coarse block $K_i$
\begin{equation}\label{eq:spectral1}
\int_{K_{i,f}} \kappa_0  \nabla \phi_{\text{aux},1}^{(ik)} \cdot \nabla v = \lambda_k^{(i)} \int_{K_{i,f}} \phi_{\text{aux},1}^{(ik)} v, 
\end{equation}
$\forall v \in V(K_{i,f})$, where $\lambda_k^{(i)} \in \mathbb{R}$ and $\phi_{\text{aux},k}^{(i)} \in V(K_i)$ are eigenvalues and eigenfuncitons of the generalized spectral problem. 

After solving the spectral problem, one can list the eigenvalues of \eqref{eq:spectral1} in an ascending order, and choose the first $l_{i,1}$ eigenfunctions $\phi_{\text{aux},1}^{(ik)}$ to form the auxiliary space. 

Then, we define $V_{\text{aux},1}^{(i)} :=  \text{span} \{\phi_{\text{aux},1}^{(ik)}, \quad 1 \leq k \leq l_{i,1}\}$, where $1\leq i \leq N_c$ and $N_c$ is the number of coarse blocks. Then the global auxiliary space $V_{\text{aux}} = \bigoplus_{i} V_{\text{aux}}^{(i)}$.

For the second auxiliary space, we solve the following:
\begin{equation}\label{eq:spectral2}
\int_{K_{i,m}} \kappa_0 \nabla  \phi_{\text{aux},2}^{(ik)}  \cdot \nabla v= \varrho_k^{(i)} \int_{K_{i,m}}    \phi_{\text{aux},2}^{(ik)}  v.
\end{equation}
Similarly, arrange the eigenvalues $\varrho_k^{(i)}$ of \eqref{eq:spectral2} in ascending order and take the first $l_{i,2}$ eigenfunctions, and we will restrict these spectral basis in the region $K_{i,m}$ to form $V_{\text{aux},2}$,
\begin{equation*}
    V_{\text{aux},2}^{(i)} = \text{span} \{\phi_{\text{aux},2}^{(ik)}, \quad 1 \leq k \leq l_{i,2}\}
\end{equation*}
 where $1\leq i \leq N_c$. The auxiliary space $V_{\text{aux},2} = \bigoplus_{i} V_{\text{aux}, 2}^{i}$.

After obtaining the auxiliary spaces, we solve the following constraint energy minimizing problem on a localized region $K_i^+$ to get the multiscale basis $\psi_{H,1}^{ij}$, 
\begin{equation}\label{eq:specbasis}
\begin{aligned}
 a(\psi_{H,1}^{ij}, w) +  s_f(w, \mu_1^{ij}) +  s_m(w, \mu_2^{ij})&= 0, \quad &&\forall w \in V(K_i^+),\\
s_f(\psi_{H,1}^{ij}, \eta ) &=  s_f(\phi_{\text{aux},1}^{ij} , \eta),   \quad &&\forall \eta \in V_{\text{aux},1}^{(i)},\\
s_m(\psi_{H,1}^{ij}, \nu) &=  0,   \quad &&\forall \nu  \in V_{\text{aux},2}^{(i)}. 
\end{aligned}
\end{equation}
where $\phi_{\text{aux},1}^{(ij)} \in  V_{\text{aux},1}^{(i)}$ is the $j$-th auxiliary basis on $K_i$, $a_f(u,v) = \int_{K_{i,f}} \kappa_0 \nabla u \cdot \nabla v$, $s_f(u,v) = \int_{K_{i,f}}  u  v$,$a_m(u,v) = \int_{K_{i,m}} \kappa \nabla u \cdot \nabla v$  $s_m(u,v) = \int_{K_{i,m}}  u  v$.

The multiscale space then forms our first subspace 
$$V_{H,1} := \text{span} \{\psi_{H,1}^{ij}, \; 1\leq j \leq  l_{i,1}, 1 \leq i \leq N_c\}$$.

To construct the subspace for $V_{H,2}$, for each auxiliary basis $\phi_{\text{aux},2}^{(ij)} \in  V_{\text{aux},2}^{i,(0)}$, we find $\psi_{H,2}^{ij}\in V(K_i^+)$, $\mu_1^{ij} \in V_{\text{aux},1}^{i}$, $\mu_2^{ij} \in V_{\text{aux},2}^{i}$ such that
\begin{equation}\label{eq:specbasis2}
\begin{aligned}
 a(\psi_{H,2}^{ij}, w) +  s_f(w, \mu_1^{ij}) + s_m(w, \mu_2^{ij})&= 0, \quad &&\forall w \in V(K_i^+),\\
 s_f(\psi_{H,2}^{ij}, \eta) &=  0,   \quad &&\forall \eta \in V_{\text{aux},1}^{(i)},\\
s_m(\psi_{H,2}^{ij}, \nu) &=  s_m(\phi_{\text{aux},2}^{ij} , \nu),   \quad &&\forall \nu \in V_{\text{aux},2}^{(i)}.
\end{aligned}
\end{equation}
Then the second multiscale space is defined as
\begin{align*}
    V_{H,2} &= \text{span} \{\psi_{H,2}^{ij}, \quad 1 \leq j \leq l_{i,2}, \;\; 1\leq i \leq N_c\}
\end{align*}


\section{Stability and convergence} \label{sec:stability}
\subsection{Stability}
In this section, we will analysis the stability of the proposed scheme \eqref{eq:partial_ein_exp1}-\eqref{eq:partial_ein_exp2}. 
%
We assume $\kappa_{u}=\kappa_{x}(x)\kappa_{u}(u)$ is Lipschitz continuous and 
\begin{align*}
\kappa_{u}(w)-\kappa_{u}(v))  \leq C_{u'}|w-v|,\;\forall w,v\in\mathbb{R},\;\;
\kappa_{u}(v)  \leq C_{u}\;\forall v\in\mathbb{R}.
\end{align*}

Let $a_{0}$ and $\tilde{a}$ be 
\begin{align*}
a_{0}(u,v)  =\int_{\Omega}\kappa_{0}\nabla u\cdot\nabla v,\;\;
\tilde{a}(w;u,v)  =a(w;u,v)-a_{0}(u,v)
\end{align*}
where $\kappa_{0}(\cdot):=\kappa_{x}(\cdot)\kappa_{u}(\tilde{u}(\cdot))$ for a given
$\tilde{u}\in H^{1}$.

Define a norm $\|\cdot\|_{\kappa}$ as $\|u\|_{\kappa}:=\Big(\int_{\Omega}\kappa_{0}|\nabla u|^{2}\Big)^{\frac{1}{2}}.$

Let $V_{H}=V_{H,1}\oplus V_{H,2}$ and $u_{H}^{n}=u_{H,1}^{n}+u_{H,2}^{n}$ be the solution of
\begin{equation}\label{eq:ein_scheme}
\begin{aligned}
(\cfrac{u_{H,1}^{n+1}-u_{H,1}^{n}+u_{H,2}^{n}-u_{H,2}^{n-1}}{\Delta t},v_1)+a_{0}(u_{H,1}^{n+1}+u_{H,2}^{n},v_1)+\tilde{a}(u_{H}^{n};u_{H}^{n},v_1) & =(f^{n},v_1),\\
(\cfrac{u_{H,2}^{n+1}-u_{H,2}^{n}+u_{H,1}^{n}-u_{H,1}^{n-1}}{\Delta t},v_2)+a_{0}(u_{H,1}^{n+1}+u_{H,2}^{n},v_2)+\tilde{a}(u_{H}^{n};u_{H}^{n},v_2) & =(f^{n},v_2).
\end{aligned}
\end{equation}
where $\forall v_1\in V_{H,1}, \forall v_2\in V_{H,2}$. We have the following stability results.
\begin{lemma}\label{lemma1}
Suppose $C_{1}:=\min_{1\leq n\leq N}\cfrac{\kappa_{u}(u_{H}^{n})-\kappa_{u}(\tilde{u})}{\kappa_{u}(\tilde{u})}<1$, $\|\kappa_{x}^{\frac{1}{2}}\nabla u_{H}^{n+1}\|_{L^{\infty}}<C_{0}$
for all $n$ and $\sup_{v_{2}\in V_{H,2}\backslash\{0\}}\cfrac{\|v_{2}\|_{\kappa}^{2}}{\|v_{2}\|^{2}}\leq\cfrac{(1-\gamma)(1-C_{1})}{\Delta t}$.
We have 
\begin{align*}
 & \Delta t^{-1}\sum_{i=1,2}\cfrac{\gamma}{2}\|(\delta u_{H,i})^{n+1}\|^{2}+\Big(1-\cfrac{D^{2}\Delta t}{2(1-\gamma)}\Big)\int_{\Omega}\cfrac{\kappa_{x}\kappa_{u}(u_{H}^{n+1})}{2}|\nabla u_{H}^{n+1}|^{2}\\
\leq & \Delta t^{-1}\sum_{i=1,2}\cfrac{\gamma}{2}\|(\delta u_{H,i})^{n}\|^{2}+\int_{\Omega}\cfrac{\kappa_{x}\kappa_{u}(u_{H}^{n})}{2}|\nabla u_{H}^{n}|^{2}+(f^{n},(\delta u_{H})^{n+1}).
\end{align*}
Moreover, 
\begin{align*}
&\max_{n}\Big(\Delta t^{-1}\sum_{i=1,2}\cfrac{\gamma}{2}\|(\delta u_{H,i})^{n}\|^{2}+\int_{\Omega}\cfrac{\kappa_{x}\kappa_{u}(u_{H}^{n})}{2}|\nabla u_{H}^{n}|^{2}\Big) \\
\leq & e^{\frac{D^{2}T}{2(1-\gamma)}}\Big(C_{T}\|f^{n}\|_{L^{\infty}}^{2}+\Delta t^{-1}\sum_{i=1,2}\cfrac{\gamma}{2}\|(\delta u_{H,i})^{0}\|^{2}+\int_{\Omega}\cfrac{\kappa_{x}\kappa_{u}(u_{H}^{0})}{2}|\nabla u_{H}^{0}|^{2}\Big).
\end{align*}
\end{lemma}

\begin{proof}

Let $v_1=(\delta u_{H,1})^{n+1}, v_2=(\delta u_{H,2})^{n+1}$ in the two equations of \eqref{eq:ein_scheme} respectively, and
\[
(\delta w)^{n+1}:=(w^{n+1}-w^{n}).
\]
We then obtain 
\begin{align*}
\Delta t^{-1}\Big(\|(\delta u_{H,1})^{n+1}\|^{2}+((\delta u_{H,2})^{n},(\delta u_{H,1})^{n+1})\Big)\\
+a_{0}(u_{H,1}^{n+1}+u_{H,2}^{n},(\delta u_{H,1})^{n+1})+\tilde{a}(u_{H}^{n};u_{H}^{n},(\delta u_{H,1})^{n+1}) & =(f^{n},(\delta u_{H,1})^{n+1}),
\end{align*}
\begin{align*}
\Delta t^{-1}\Big(\|(\delta u_{H,2})^{n+1}\|^{2}+((\delta u_{H,1})^{n},(\delta u_{H,2})^{n+1})\Big)\\
+a_{0}(u_{H,1}^{n+1}+u_{H,2}^{n},(\delta u_{H,2})^{n+1})+\tilde{a}(u_{H}^{n};u_{H}^{n},(\delta u_{H,2})^{n+1}) & =(f^{n},(\delta u_{H,1})^{n+1}).
\end{align*}
Since $|(v_{1},v_{2})|\leq\gamma\|v_{1}\|\|v_{2}\|$ for any $v_{1}\in V_{H,1}, v_{2}\in V_{H,2}$,
we obtain
\begin{align*}
&|((\delta u_{H,2})^{n},(\delta u_{H,1})^{n+1})|+|((\delta u_{H,1})^{n},(\delta u_{H,2})^{n+1})| \\
\leq & \gamma\Big(\|(\delta u_{H,1})^{n+1}\|\|(\delta u_{H,2})^{n}\|+\|(\delta u_{H,2})^{n+1}\|\|(\delta u_{H,1})^{n}\|\Big)\\
\leq & \cfrac{\gamma}{2}\sum_{j=0,1}\sum_{i=1,2}\|(\delta u_{H,i})^{n+j}\|^{2}.
\end{align*}
Therefore, we have 
\begin{equation}\label{eq:est1}
   \begin{aligned}
\Delta t^{-1}\sum_{i=1,2}\Big(\Big(1-\cfrac{\gamma}{2}\Big)\|(\delta u_{H,i})^{n+1}\|^{2}-\cfrac{\gamma}{2}\|(\delta u_{H,i})^{n}\|^{2}\Big)\\
+a_{0}(u_{H,1}^{n+1}+u_{H,2}^{n},(\delta u_{H})^{n+1})+\tilde{a}(u_{H}^{n};u_{H}^{n},(\delta u_{H})^{n+1}) & \leq(f^{n},(\delta u_{H})^{n+1}).
\end{aligned} 
\end{equation}

We denote $\kappa_{x}\kappa_{u}(u_{H}^{n})$ as $\kappa^{n}$ and write
\begin{align*}
 & a_{0}(u_{H,1}^{n+1}+u_{H,2}^{n},(\delta u_{H})^{n+1})+\tilde{a}(u_{H}^{n};u_{H}^{n},(\delta u_{H})^{n+1} )\\
 = &\int_{\Omega}\kappa_{0}\nabla u_{H}^{n+1}\cdot\nabla(\delta u_{H})^{n+1}-\int_{\Omega}\kappa_{0}\nabla(\delta u_{H,2})^{n+1}\cdot\nabla(\delta u_{H})^{n+1} \\
 &+  \int_{\Omega}(\kappa^{n}-\kappa_{0})\nabla u_{H}^{n}\cdot\nabla(\delta u_{H})^{n+1}
\end{align*}
We have 
\begin{align*}
\int_{\Omega}\kappa_{0}\nabla (u_{H}^{n+1}\cdot\nabla(\delta u_{H})^{n+1}= & \int_{\Omega}\cfrac{\kappa_{0}}{2}\Big(|\nabla u_{H}^{n+1}|^{2}-|\nabla u_{H}^{n}|^{2}+|\nabla(\delta u_{H})^{n+1}|^{2}\Big),\\
\int_{\Omega}(\kappa^{n}-\kappa_{0})\nabla u_{H}^{n}\cdot\nabla(\delta u_{H})^{n+1} 
 & =\int_{\Omega}(\cfrac{\kappa^{n}-\kappa_{0}}{2})\Big(|\nabla u_{H}^{n+1}|^{2}-|\nabla u_{H}^{n}|^{2}-|\nabla(\delta u_{H})^{n+1}|^{2}\Big).
\end{align*}
Summing these two equations, we obtain
\begin{align*}
 & \int_{\Omega}\kappa_{0}\nabla u_{H}^{n+1}\cdot\nabla(\delta u_{H})^{n+1}+\int_{\Omega}(\kappa^{n}-\kappa_{0})\nabla u_{H}^{n}\cdot\nabla(\delta u_{H})^{n+1}\\
= & \int_{\Omega}\cfrac{\kappa^{n}}{2}\Big(|\nabla u_{H}^{n+1}|^{2}-|\nabla u_{H}^{n}|^{2}\Big)+\int_{\Omega}\cfrac{2\kappa_{0}-\kappa^{n}}{2}|\nabla(\delta u_{H})^{n+1}|^{2}\Big)
\end{align*}
Since
\begin{align*}
&\int_{\Omega}\cfrac{\kappa^{n}}{2}\Big(|\nabla u_{H}^{n+1}|^{2}-|\nabla u_{H}^{n}|^{2}\Big)\\
 =&\int_{\Omega}\Big(\cfrac{\kappa^{n+1}}{2}|\nabla u_{H}^{n+1}|^{2}-\cfrac{\kappa^{n}}{2}|\nabla u_{H}^{n}|^{2}\Big)+\int_{\Omega}\cfrac{\kappa^{n+1}-\kappa^{n}}{2}|\nabla u_{H}^{n+1}|^{2}.
\end{align*}
and the term $\int_{\Omega}\cfrac{\kappa^{n+1}-\kappa^{n}}{2}|\nabla u_{H}^{n+1}|^{2}$ can be estimated by
\begin{align*}
&\Big|\int_{\Omega}\cfrac{\kappa^{n+1}-\kappa^{n}}{2}|\nabla u_{H}^{n+1}|^{2}\Big|  =\cfrac{1}{2}\Big|\int_{\Omega}\kappa_{x}(\kappa_{u}(u_{H}^{n+1})-\kappa_{u}(u_{H}^{n}))|\nabla u_{H}^{n+1}|^{2}\Big|\\
 \leq & \cfrac{C_{u'}}{2}\int_{\Omega}\kappa_{x}|(\delta u_{H})^{n+1}||\nabla u_{H}^{n+1}|^{2}\leq D\|(\delta u_{H})^{n+1}\|_{L^{2}}\|(\kappa^{n+1})^{\frac{1}{2}}\nabla u_{H}^{n+1}\|_{L^{2}}
\end{align*}
where $D=C_{u'}C_{0}\underline{\kappa}_{u}^{-1}$. 

Putting these estimates in \eqref{eq:est1}, we obtain
\begin{equation}\label{eq:est2}
    \begin{aligned}
 & \Delta t^{-1}\sum_{i=1,2}\Big(\cfrac{\gamma}{2}\|(\delta u_{H,i})^{n+1}\|^{2}+(1-\gamma)\|(\delta u_{H,i})^{n+1}\|^{2}\Big)+\int_{\Omega}\cfrac{\kappa^{n+1}}{2}|\nabla u_{H}^{n+1}|^{2} \\
 &+\int_{\Omega}\cfrac{2\kappa_{0}-\kappa^{n}}{2}|\nabla(\delta u_{H})^{n+1}|^{2}-\int_{\Omega}\Big(\kappa_{0}\nabla(\delta u_{H,2})^{n+1}\Big)\cdot\nabla(\delta u_{H})^{n+1}\\
\leq & \Delta t^{-1}\sum_{i=1,2}\cfrac{\gamma}{2}\|(\delta u_{H,i})^{n}\|^{2}+\int_{\Omega}\cfrac{\kappa^{n}}{2}|\nabla u_{H}^{n}|^{2}+(f^{n},\delta u_{H}^{n+1})+\\
& D\|(\delta u_{H})^{n+1}\|_{L^{2}}\|(\kappa^{n+1})^{\frac{1}{2}}\nabla u_{H}^{n+1}\|_{L^{2}}\\
\leq & \Delta t^{-1}\sum_{i=1,2}\cfrac{\gamma}{2}\|(\delta u_{H,i})^{n}\|^{2}+\int_{\Omega}\cfrac{\kappa^{n}}{2}|\nabla u_{H}^{n}|^{2}+(f^{n},\delta u_{H}^{n+1})+\cfrac{1-\gamma}{2\Delta t}\|(\delta u_{H})^{n+1}\|_{L^{2}}^{2}\\
&+\cfrac{D^{2}\Delta t}{2(1-\gamma)}\int_{\Omega}\cfrac{\kappa^{n+1}}{2}|\nabla u_{H}^{n+1}|^{2}.
\end{aligned}
\end{equation}

Since $|\cfrac{\kappa_{0}-\kappa^{n}}{\kappa_{0}}|\leq C_{1}$, we
have $\cfrac{2\kappa_{0}-\kappa^{n}}{2}\geq\cfrac{1-C_{1}}{2}\kappa_{0}$. Thus
\begin{align*}
 & \int_{\Omega}\cfrac{2\kappa_{0}-\kappa^{n}}{2}|\nabla(\delta u_{H})^{n+1}|^{2}-\int_{\Omega}\Big(\kappa_{0}\nabla(\delta u_{H,2})^{n+1}\Big)\cdot\nabla(\delta u_{H})^{n+1}\\
\geq & \cfrac{-1}{2(1-C_{1})}\int_{\Omega}\kappa_{0}|\nabla(\delta u_{H,2})^{n+1}|^{2}.
\end{align*}

Furthermore, by $\cfrac{\|(\delta u_{H,2})^{n+1}\|_{\kappa}^{2}}{\|(\delta u_{H,2})^{n+1}\|^{2}}\leq\cfrac{(1-\gamma)(1-C_{1})}{\Delta t}$,
we have
\begin{align*}
 & \Delta t^{-1}\sum_{i=1,2}\Big(\cfrac{\gamma}{2}\|(\delta u_{H,i})^{n+1}\|^{2}\Big)+\Big(1-\cfrac{D^{2}\Delta t}{2(1-\gamma)}\Big)\int_{\Omega}\cfrac{\kappa^{n+1}}{2}|\nabla u_{H}^{n+1}|^{2}\\
\leq & \Delta t^{-1}\sum_{i=1,2}\Big(\cfrac{\gamma}{2}\|(\delta u_{H,i})^{n+1}\|^{2}+\cfrac{(1-\gamma)}{2}\|(\delta u_{H,i})^{n+1}\|^{2}\Big) \\
 & +\Big(1-\cfrac{D^2 \Delta t}{2(1-\gamma)}\Big)\int_{\Omega}\cfrac{\kappa^{n+1}}{2}|\nabla u_{H}^{n+1}|^{2}-\cfrac{1}{2(1-C_{1})}\int_{\Omega}\kappa_{0}|\nabla(\delta u_{H,2})^{n+1}|^{2}\\
\leq & \Delta t^{-1}\sum_{i=1,2}\Big(\cfrac{\gamma}{2}\|(\delta u_{H,i})^{n}\|^{2}\Big)+\int_{\Omega}\cfrac{\kappa^{n}}{2}|\nabla u_{H}^{n}|^{2}+(f^{n},(\delta u_{H})^{n+1}),
\end{align*}
where \eqref{eq:est2} is used in the last inequality.

Finally, by Gronwall's inequality, we obtain
\begin{align*}
&\max_{n}\Big(\Delta t^{-1}\sum_{i=1,2}\Big(\cfrac{\gamma}{2}\|(\delta u_{H,i})^{n}\|^{2}\Big)+\int_{\Omega}\cfrac{\kappa_{x}\kappa_{u}(u_{H}^{n})}{2}|\nabla u_{H}^{n}|^{2}\Big) \\
\leq & e^{\frac{D^{2}T}{2(1-\gamma)}}\Big(C_{T}\|f^{n}\|_{L^{\infty}}^{2}+\Delta t^{-1}\sum_{i=1,2}\Big(\cfrac{\gamma}{2}\|(\delta u_{H,i})^{0}\|^{2}\Big)+\int_{\Omega}\cfrac{\kappa_{x}\kappa_{u}(u_{H}^{0})}{2}|\nabla u_{H}^{0}|^{2}\Big).
\end{align*}
This completes the proof.
\end{proof}

\subsection{Convergence}
In this subsection, we will discuss the convergence of the proposed scheme. We first define an elliptic projection operator $P_{H}:H^{1}(\Omega)\rightarrow V_{H}$
such that 
\[
a(P_{H}(v),w)=a(v,w)\;\forall w\in V_{H}.
\]
Since $V_{H}=V_{H,1}+V_{H,2}$, we can introduce $P_{H,i}:H^{1}(\Omega)\rightarrow V_{H,i}$
such that $P_{H}=P_{H,1}+P_{H,2}$.
Let $u^n = u(\cdot, t^n)$. Define some residual terms  as follows, $r_{\delta}^{n},r_{\delta^{2},i}^{n},r_{a}^{n}$ and $r_{\delta a,i}^{n}$ for latter analysis. 
\begin{align*}
&r_{\delta}^{n}  :=\partial_{t}u^{n}-\cfrac{P_H(u^{n+1}-u^{n})}{\Delta t}, \;\; r_{\delta^{2},i}^{n}:=\Delta tP_{H,i}\Big(\cfrac{(-u^{n+1}+2u^{n}-u^{n-1})}{\Delta t^{2}}\Big),\\
&(r_{a}^{n},v)  :=\tilde{a}(u^{n};u^{n},v)-\tilde{a}(P_H u^{n};P_H u^{n},v),
\;\; (r_{\delta a,i}^{n},v)  :=\Delta ta_{0}(P_{H.i}(\cfrac{u^{n+1}-u^{n}}{\Delta t}),v).
\end{align*}
We also define the error terms $\eta_{i}^{n}=P_{H,i}u^{n}-u_{H,i}^{n}$ and $\eta^{n}=\sum_{i=1,2}\eta_{i}^{n}$. 
We will present the stability of the discrete error $\eta_{i}^{n}$ in energy norm and $L_2$ norm in the following two lemmas.
\begin{lemma} \label{lem:error_a}
	Suppose $\|\kappa_{x}^{\frac{1}{2}}\nabla u_{H}^{n+1}\|_{L^{\infty}}<C_{0}$ for all $n$, $\|P_H(\partial_t u)\|_{L^{\infty}([0,T],L^{\infty})}<C_{0}$
	and $\sup_{v_{2}\in V_{H,2}\backslash\{0\}}\cfrac{\|v_{2}\|_{\kappa}^{2}}{\|v_{2}\|^{2}}\leq\cfrac{(1-\gamma)(1-C_{1})}{2\Delta t}$.
	There exist constant $F_{1}>0$ such that 
	\begin{align*}
		& \sum_{i=1,2}\cfrac{1+\gamma}{4\Delta t}\|\delta\eta_{i}^{n+1}\|^{2}+(\cfrac{1-C_{1}}{8})\|\delta\eta^{n+1}\|_{\kappa}^{2}+\Big(1-\cfrac{D^{2}\Delta t}{(1-\gamma)}\Big)\|\eta^{n+1}\|^2_{\kappa(u^{n+1})}\\
		\leq & \sum_{i=1,2}\cfrac{\gamma}{2\Delta t}\|\delta\eta_{i}^{n}\|^{2}+\|\eta^n\|^2_{\kappa(u^n)}+F_{1}\|\eta^{n+1}\|^{2}+(r_{1}^{n},\delta\eta_{1}^{n+1})+(r_{2}^{n},\delta\eta_{2}^{n+1})
	\end{align*}
	where $r_{1}^{n}:=-r_{\delta}^{n}+r_{\delta^{2},2}^{n}-r_{\delta a,2}^{n}-r_{a}^{n}$,
	$r_{2}^{n}:=-r_{\delta}^{n}+r_{\delta^{2},1}^{n}-r_{\delta a,2}^{n}-r_{a}^{n}$ and 
	\[
	\|\eta^n\|^2_{\kappa(u^n)}:=\int_{\Omega}(\cfrac{\kappa_{x}\kappa_{u}(P_{H}u^{n})}{2})|\nabla\eta^{n}|^{2}.
	\]
\end{lemma}

	
\begin{lemma} \label{lem:error_L2}
	Suppose $\|\kappa_{x}^{\frac{1}{2}}\nabla u_{H}^{n+1}\|_{L^{\infty}}<C_{0}$
	for all $n$ and $\sup_{v_{2}\in V_{H,2}\backslash\{0\}}\cfrac{\|v_{2}\|_{\kappa}^{2}}{\|v_{2}\|^{2}}\leq\cfrac{(1-\gamma)(1-C_{1})}{2\Delta t}$.
	There exist constants $C_{2}>0,F_{2}>0$ such that 
	\begin{align*}
		& \tilde{E}_{n+1}(\eta)+\cfrac{1}{2\Delta t}\|\delta\eta^{n}\|^{2}+\cfrac{1}{2}\int_{\Omega}\kappa_{x}\kappa_{u}(P_{H}u^{n})|\nabla\eta^{n}|^{2}\\
		\leq & \tilde{E}_{n}(\eta)+F_{2}\|\eta^{n+1}\|^{2}+\cfrac{1}{2}\Big(\|\delta\eta^{n+1}\|_{\kappa}^{2}+C_{2}\|\delta\eta_{2}^{n+1}\|_{\kappa}^{2}\Big)\\
		&+\cfrac{1}{2\Delta t}\sum_{i=1,2}\Big(\|\delta\eta_{i}^{n+1}\|^{2}+\|\delta\eta_{i}^{n}\|^{2}\Big)+(r_{1}^{n},\eta_{1}^{n})+(r_{2}^{n},\eta_{2}^{n})
	\end{align*}
	where $\tilde{E}_{n+1}(\eta):={E}_{n+1}(\eta)+\cfrac{1}{2}\|\eta^{n+1}\|_{\kappa}^{2}$ and $ {E}_{n+1}(\eta) = \cfrac{1}{2\Delta t}\Big(\sum_{i=1,2}(\|\eta_{i}^{n+1}\|^{2}-\|\eta_{i}^{n}\|^{2})+\|\eta^{n}\|^{2}\Big).$
\end{lemma}

Combining Lemma \ref{lem:error_a} and Lemma \ref{lem:error_L2}, we now show the convergence of the discrete error in the following lemma.
\begin{theorem}\label{lemma_convegence}
	Suppose $\|\kappa_{x}^{\frac{1}{2}}\nabla u_{H}^{n+1}\|_{L^{\infty}}<C_{0}$ for all $n$, $\|P_H(\partial_t u)\|_{L^{\infty}([0,T],L^{\infty})}<C_{0}$ and $\sup_{v_{2}\in V_{H,2}\backslash\{0\}}\cfrac{\|v_{2}\|_{\kappa}^{2}}{\|v_{2}\|^{2}}\leq\cfrac{(1-\gamma)}{2\Delta t}\min\{1-C_{1},\cfrac{C_{2}}{2}\}$
	and $\alpha=\cfrac{1}{4}\min\{1-\gamma,1-C_{1}\}$, we have 
	\begin{align*}
		& \max_{n}\Big(\cfrac{1}{2\Delta t}\sum_{i=1,2}G_{\alpha}\|\eta_{i}^{n+1}\|^{2}+\int_{\Omega}(\cfrac{\kappa_{x}\kappa_{u}(P_{H}u^{n+1})}{2})|\nabla\eta^{n+1}|^{2}\Big)\leq\max_{n}\Big(\tilde{E}_{n+1,\alpha}(\eta)\Big)\\
		\leq & e^{C_{3}T}\Big(C_{T}\max_{n}(\|r^n_{\delta }\|^{2}+\|r_{a}^{n}\|^2_{\kappa^{*}}+\|\cfrac{\delta r_{a}^{n}}{\Delta t}\|^2_{\kappa^{*}}+\sum_{i=1,2}(\|r_{\delta^{2},i}^{n}\|^2+\|r_{\delta a,i}^{n}\|^2_{\kappa^{*}}+\|\cfrac{\delta r_{\delta a,i}^{n}}{\Delta t}\|^2_{\kappa^{*}}))\\
		&+\sum_{i=1,2}\Big(\cfrac{\gamma}{2\Delta t}+\cfrac{1-\gamma}{8\Delta t}\Big)\|\delta\eta_{i}^{1}\|^{2}+\int_{\Omega}(\cfrac{\kappa_{x}\kappa_{u}(P_{H}u^{1})}{2})|\nabla\eta^{1}|^{2}+\alpha\Big(E_{1}(\eta)+\cfrac{1}{2}\|\eta^{1}\|_{\kappa}^{2}\Big)\Big).
	\end{align*}
	where $\tilde{E}_{n+1,\alpha}(\eta):=\sum_{i=1,2}\Big(\cfrac{\gamma}{2\Delta t}+\cfrac{1-\gamma}{8\Delta t}\Big)\|\delta\eta_{i}^{n+1}\|^{2}+\int_{\Omega}(\cfrac{\kappa_{x}\kappa_{u}(P_{H}u^{n+1})}{2})|\nabla\eta^{n+1}|^{2}+\alpha\Big(E_{n+1}(\eta)+\cfrac{1}{2}\|\eta^{n+1}\|_{\kappa}^{2}\Big)$ and 
 $\|r\|_{\kappa^*} := \sup_{v\in V}\cfrac{(r,v)}{\|v\|_{\kappa}}$.
\end{theorem}

\begin{proof}
	We will start the proof from estimating $\|\eta^{n+1}\|^{2}$. Since 
 \[\|\eta^{n}\|^{2}=\sum_{i=1,2}\|\eta_{i}^{n}\|^{2}-2(\eta_{1}^{n},\eta_{2}^{n})\geq(1-\gamma)\sum_{i=1,2}\|\eta_{i}^{n}\|^{2},
 \]
	we have
	$
	E_{n+1}(\eta)\geq\cfrac{1}{2\Delta t}\Big(\sum_{i=1,2}(\|\eta_{i}^{n+1}\|^{2}-\gamma\|\eta_{i}^{n}\|^{2})\Big).$
	
 Thus, we have 
	\begin{align*}
		& \alpha E_{n+1}(\eta)+\cfrac{\gamma}{2\Delta t}\sum_{i=1,2}\|\delta\eta_{i}^{n+1}\|^{2}\\
		\geq& \cfrac{1}{2\Delta t}\Big(\sum_{i=1,2}(\alpha+\gamma)\|\eta_{i}^{n+1}\|^{2}+\gamma(1-\alpha)\|\eta_{i}^{n}\|^{2}-2\gamma(\eta_{i}^{n+1},\eta_{i}^{n})\Big)\\
		\geq & \cfrac{1}{2\Delta t}\Big(\sum_{i=1,2}(\alpha+\gamma-\cfrac{\gamma}{1-\alpha})\|\eta_{i}^{n+1}\|^{2}\Big)
	\end{align*}
	We take $\alpha<1-\gamma$ and obtain
 $\alpha+\gamma-\cfrac{\gamma}{1-\alpha}=\alpha\cfrac{1-\alpha-\gamma}{1-\alpha}=:G_{\alpha}>0$ and 
  \begin{align*}
		& \cfrac{4\Delta t}{G_{\alpha}}\Big(\alpha E_{n+1}(\eta)+\cfrac{\gamma}{2\Delta t}\sum_{i=1,2}\|\delta\eta_{i}^{n+1}\|^{2}\Big)
		\geq  \cfrac{2}{G_{\alpha}}\Big(\sum_{i=1,2}G_{\alpha}\|\eta_{i}^{n+1}\|^{2}\Big)\geq\|\eta^{n+1}\|^{2}.
	\end{align*}
 
	Furthermore, combining the results from Lemma \ref{lem:error_a} and Lemma \ref{lem:error_L2}, we get
	\begin{align*}
		& \sum_{i=1,2}\Big(\cfrac{\gamma}{2\Delta t}+\cfrac{1-\gamma}{4\Delta t}\Big)\|\delta\eta_{i}^{n+1}\|^{2}+\cfrac{1-C_{1}}{8}\|\delta\eta^{n+1}\|_{\kappa}^{2}+\alpha E_{n+1}(\eta)+\cfrac{\alpha}{2}\|\eta^{n+1}\|_{\kappa}^{2}\\
		&+\cfrac{\alpha}{2\Delta t}\|\delta\eta^{n}\|^{2} +\Big(1-D_1\Delta t\Big)\int_{\Omega}\cfrac{\kappa_{x}\kappa_{u}(P_{H}u^{n+1})}{2}|\nabla\eta^{n+1}|^{2} +\cfrac{\alpha}{4}\int_{\Omega}\kappa_{x}\kappa_{u}(P_{H}u^{n})|\nabla\eta^{n}|^{2}
  	\end{align*}
   	\begin{align*}
		\leq & \sum_{i=1,2}\cfrac{\gamma}{2\Delta t}\|\delta\eta_{i}^{n}\|^{2}+\int_{\Omega}(\cfrac{\kappa_{x}\kappa_{u}(P_{H}u^{n})}{2})|\nabla\eta^{n}|^{2}+(F_{1}+\alpha F_2)\|\eta^{n+1}\|^{2}\\
		& +\alpha\Big(E_{n}(\eta)+\cfrac{1}{2}\|\eta^{n}\|_{\kappa}^{2}+\cfrac{C_{2}}{2}\|\delta\eta_{2}^{n+1}\|_{\kappa}^{2}+(r_{1}^{n},\eta_{1}^{n})+(r_{2}^{n},\eta_{2}^{n})\Big)\\
  &+\cfrac{\alpha}{2\Delta t}\sum_{i=1,2}\Big(\|\delta\eta_{i}^{n+1}\|^{2}+\|\delta\eta_{i}^{n}\|^{2}\Big)+(r_{1}^{n},\delta\eta_{1}^{n+1})+(r_{2}^{n},\delta\eta_{2}^{n+1}).
	\end{align*}
	For $\alpha=\cfrac{1}{4}\min\{1-\gamma,1-C_{1}\}$, using $\|\delta\eta^{n}\|^{2}>(1-\gamma)\sum_{i=1,2}\|\delta\eta_{i}^{n+1}\|^{2}$,
	we have
	\begin{align*}
		& \sum_{i=1,2}\Big(\cfrac{\gamma}{2\Delta t}+\cfrac{1-\gamma}{8\Delta t}\Big)\|\delta\eta_{i}^{n+1}\|^{2}+\Big(1-D_1\Delta t\Big)\int_{\Omega}(\cfrac{\kappa_{x}\kappa_{u}(P_{H}u^{n+1})}{2})|\nabla\eta^{n+1}|^{2}\\
		& +\alpha\Big(E_{n+1}(\eta)+\cfrac{1}{2}\|\eta^{n+1}\|_{\kappa}^{2}+\cfrac{1-\gamma}{2\Delta t}\sum_{i=1,2}\|\delta\eta_{i}^{n+1}\|^{2}\Big) \\
		\leq & \sum_{i=1,2}\Big(\cfrac{\gamma}{2\Delta t}+\cfrac{1-\gamma}{8\Delta t}\Big)\|\delta\eta_{i}^{n}\|^{2}+(1-\cfrac{\alpha}{2})\int_{\Omega}(\cfrac{\kappa_{x}\kappa_{u}(P_{H}u^{n})}{2})|\nabla\eta^{n}|^{2}\\
		& +\alpha\Big(E_{n}(\eta)+\cfrac{1}{2}\|\eta^{n}\|_{\kappa}^{2}+\cfrac{C_{2}}{2}\|\delta\eta_{2}^{n+1}\|_{\kappa}^{2}+(r_{1}^{n},\eta_{1}^{n})+(r_{2}^{n},\eta_{2}^{n})\Big)\\
		& +(r_{1}^{n},\delta\eta_{1}^{n+1})+(r_{2}^{n},\delta\eta_{2}^{n+1})+(F_{1}+\alpha F_2)\|\eta^{n+1}\|^{2}.
	\end{align*}
	Since $\sup_{v_{2}\in V_{H,2}\backslash\{0\}}\cfrac{\|v_{2}\|_{\kappa}^{2}}{\|v_{2}\|^{2}}\leq\cfrac{(1-\gamma)}{2\Delta t}\min\{1-C_{1},\cfrac{C_{2}}{2}\}$
	, we have 
	\begin{align*}
		& \tilde{E}_{n+1,\alpha}(\eta)-D_1\Delta t\int_{\Omega}(\cfrac{\kappa_{x}\kappa_{u}(P_{H}u^{n+1})}{2})|\nabla\eta^{n+1}|^{2}\\
		\leq & \tilde{E}_{n,\alpha}(\eta)+(F_{1}+\alpha F_2)\|\eta^{n+1}\|^{2}+(r_{1}^{n},\delta\eta_{1}^{n+1})+(r_{2}^{n},\delta\eta_{2}^{n+1})+\alpha\Big((r_{1}^{n},\eta_{1}^{n})+(r_{2}^{n},\eta_{2}^{n})\Big)\\
		\leq & \tilde{E}_{n,\alpha}(\eta)+\cfrac{4\Delta t(F_{1}+\alpha F_2)}{G_{\alpha}}\Big(\alpha E_{n+1}(\eta)+\cfrac{\gamma}{2\Delta t}\sum_{i=1,2}\|\delta\eta_{i}^{n+1}\|^{2}\Big)\\
		& +(r_{1}^{n},\delta\eta_{1}^{n+1})+(r_{2}^{n},\delta\eta_{2}^{n+1})+\alpha\Big((r_{1}^{n},\eta_{1}^{n})+(r_{2}^{n},\eta_{2}^{n})\Big).
	\end{align*}
	where $
	\tilde{E}_{n+1,\alpha}(\eta):=\sum_{i=1,2}\Big(\cfrac{\gamma}{2\Delta t}+\cfrac{1-\gamma}{8\Delta t}\Big)\|\delta\eta_{i}^{n+1}\|^{2}+\int_{\Omega}(\cfrac{\kappa_{x}\kappa_{u}(P_{H}u^{n+1})}{2})|\nabla\eta^{n+1}|^{2}+\alpha\Big(E_{n+1}(\eta)+\cfrac{1}{2}\|\eta^{n+1}\|_{\kappa}^{2}\Big)
	$.
 
	Since 
 \begin{align*}
     \Big(\alpha E_{n+1}(\eta)+\cfrac{\gamma}{2\Delta t}\sum_{i=1,2}\|\delta\eta_{i}^{n+1}\|^{2}\Big)&\leq\tilde{E}_{n+1,\alpha}(\eta),\\
     \int_{\Omega}(\cfrac{\kappa_{x}\kappa_{u}(P_{H}u^{n+1})}{2})|\nabla\eta^{n+1}|^{2}&\leq\tilde{E}_{n+1,\alpha}(\eta).
 \end{align*}
	we obtain 
	\begin{align*}
		& (1-C_{3}\Delta t)\tilde{E}_{n+1,\alpha}(\eta)
		\leq \tilde{E}_{n,\alpha}(\eta)+(r_{1}^{n},\delta\eta_{1}^{n+1})+(r_{2}^{n},\delta\eta_{2}^{n+1})+\alpha((r_{1}^{n},\eta_{1}^{n})+(r_{2}^{n},\eta_{2}^{n})).
	\end{align*}
 By the definition of $r_{1}^{n}$ and discrete integration by parts, we have
\begin{align*}
&\Big|\sum^{N}_{n=1} (r_{1}^{n},\delta\eta^{n+1}_1)\Big| \\
=& \Big|\sum^{N}_{n=1}\Big(-(r_{\delta}^{n},\delta\eta^{n+1}_1)+(r_{\delta^{2},2}^{n},\delta\eta^{n+1}_1)-(r_{\delta a,2}^{n},\delta\eta^{n+1}_1)-(r_{a}^{n},\delta\eta^{n+1}_1)\Big)\Big|\\
\leq& \Big|\sum^{N}_{n=1}\Big(\Big(\|r_{\delta}^{n}\| + \|r_{\delta^{2},2}^{n}\|\Big) \|\delta \eta^{n+1}_1\| + \Delta t\Big(\|\cfrac{\delta r_{\delta a,2}^{n}}{\Delta t}\|_{\kappa^*} + \|\cfrac{\delta r_{a}^{n}}{\Delta t}\|_{\kappa^*}\Big)\|\eta^{n+1}_1\|_{\kappa}\Big|\\
&+\Big(\|r_{\delta a,2}^{N+1}\|_{\kappa^*}+\|r_{a}^{N+1}\|_{\kappa^*}\Big)\|\eta_{1}^{N+1}\|_{\kappa} + \Big(\|r_{\delta a,2}^{1}\|_{\kappa^*}+\|r_{a}^{1}\|_{\kappa^*}\Big)\|\eta_{1}^{1}\|_{\kappa}.
\end{align*}
	We can obtain a similar inequality for $(r_{2}^{n},\delta\eta^{n+1}_2)$. By Gronwall's inequality, we obtain
	\begin{align*}
		& \max_{n}\Big(\cfrac{1}{2\Delta t}\sum_{i=1,2}G_{\alpha}\|\eta_{i}^{n+1}\|^{2}+\int_{\Omega}(\cfrac{\kappa_{x}\kappa_{u}(P_{H}u^{n+1})}{2})|\nabla\eta^{n+1}|^{2}\Big)\leq\max_{n}\Big(\tilde{E}_{n+1,\alpha}(\eta)\Big)\\
		\leq & e^{C_{3}T}\Big(C_{T}\max_{n}(\|r^n_{\delta }\|^{2}+\|r_{a}^{n}\|^2_{\kappa^{*}}+\|\cfrac{\delta r_{a}^{n}}{\Delta t}\|^2_{\kappa^{*}}+\sum_{i=1,2}(\|r_{\delta^{2},i}^{n}\|^2+\|r_{\delta a,i}^{n}\|^2_{\kappa^{*}}+\|\cfrac{\delta r_{\delta a,i}^{n}}{\Delta t}\|^2_{\kappa^{*}}))\\
		&+\sum_{i=1,2}\Big(\cfrac{\gamma}{2\Delta t}+\cfrac{1-\gamma}{8\Delta t}\Big)\|\delta\eta_{i}^{1}\|^{2}+\int_{\Omega}(\cfrac{\kappa_{x}\kappa_{u}(P_{H}u^{1})}{2})|\nabla\eta^{1}|^{2}+\alpha\Big(E_{1}(\eta)+\cfrac{1}{2}\|\eta^{1}\|_{\kappa}^{2}\Big)\Big).
	\end{align*}
 This completes the proof.
\end{proof}

Finally, we present the estimates for the residual terms appeared in Theorem \ref{lemma_convegence}.
\begin{lemma}
	Suppose $u\in C^{2}([0,T],\kappa)\cap C^{1}([0,T],\kappa)\cap L^{\infty}([0,T],W^{1,\infty}(\Omega))$, then
	\begin{align*}
		&\|r_{\delta}^{n}\|  \leq C\Big(H\|u\|_{C^{1}([0,T],\kappa)}+\Delta t\|u\|_{C^{2}([0,T],\kappa)}\Big), \;\;
		\|r_{\delta^{2},i}^{n}\|  \leq C\Delta t\|u\|_{C^{2}([0,T],\kappa)},\\
		&\|r_{a}^{n}\|_{\kappa^{*}}  \leq C(1+H\|\nabla u\|_{L^{\infty}([0,T],L^{\infty})})\|u^{n}-P_{H}u^{n}\|_{\kappa},\;\;
		\|r_{\delta a,i}^{n}\|_{\kappa^{*}}  \leq C\Delta t\|\partial_{t}u(\xi)\|_{\kappa}.
	\end{align*}
 Moreover, if $\|P_{H}\Big(\partial_{t}\kappa(u)\Big)\|_{L^{\infty}([0,T],L^{\infty})}<\infty$
and $\|\kappa_{x}^{\frac{1}{2}}\nabla u_{t}\|_{L^{\infty}([0,T],L^{\infty})}<\infty$,
we have 
\begin{align*}
\|\cfrac{\delta r_{a}^{n}}{\Delta t}\|_{\kappa^{*}}\leq & C\Big(\|\kappa_{x}^{\frac{1}{2}}\nabla u\|_{L^{\infty}([0,T],L^{\infty})}\|(I-P_{H})\partial_{t}\kappa(u)\|_{L^{\infty}([0,T],L^{2})}\\
 &+\|(I-P_{H})u\|_{L^{\infty}([0,T],\kappa)} + \|(I-P_{H})\partial_{t}u\|_{L^{\infty}([0,T],\kappa)}\\
 &  +\|(I-P_{H})\kappa(u)\|_{L^{\infty}([0,T],\kappa)}\Big)\\
\|\cfrac{\delta r_{a,i}^{n}}{\Delta t}\|_{\kappa^{*}}\leq & C\Delta t\|\partial_{tt}u\|_{C^{2}([0,T],\kappa)}
\end{align*}
\end{lemma}

\begin{proof}
	By using the properties of $P_{H}$, we have 
	\begin{align*}
		\|\kappa_{x}^{\frac{1}{2}}(I-P_{H})v\| \leq CH\|v\|_{\kappa}, &\;\;\|P_{H,i}v\| \leq C\|v\|_{\kappa}.
	\end{align*}
	We first estimate $r_{\delta}^{n}:=\partial_{t}u^{n}-\cfrac{P_{H}(u^{n+1}-u^{n})}{\Delta t}$
	and have 
	\begin{align*}
		& \|\partial_{t}u^{n}-\cfrac{P_{H}(u^{n+1}-u^{n})}{\Delta t}\|
		= \|\partial_{t}u^{n}-P_{H}(\partial_{t}u^{n})+P_{H}\Big(\partial_{t}u^{n}-\cfrac{u^{n+1}-u^{n}}{\Delta t}\Big)\|\\
		\leq & H\|\partial_{t}u^{n}\|_{\kappa}+\Delta t\|P_{H}\Big(\partial_{tt}u(\xi)\Big)\|
		\leq H\|\partial_{t}u^{n}\|_{\kappa}+C\Delta t\|\partial_{tt}u(\xi)\|_{\kappa}
	\end{align*}
	
	We next estimate $r_{\delta^{2},i}^{n}:=\Delta tP_{H,i}\Big(\cfrac{(-u^{n+1}+2u^{n}-u^{n-1})}{\Delta t^{2}}\Big)$
	and have
	
	\begin{align*}
		\|\Delta tP_{H,i}\Big(\cfrac{(-u^{n+1}+2u^{n}-u^{n-1})}{\Delta t^{2}}\Big)\|
		=  \Delta t\|P_{H,i}\Big(\partial_{tt}u(\xi)\Big)\|
		\leq  C\Delta t\|\partial_{tt}u(\xi)\|_{\kappa}.
	\end{align*}
	
	We then estimate $\|r_{a}^{n}\|_{\kappa^{*}}$ and have
	\begin{align*}
		(r_{a}^{n},v)= & \tilde{a}(u^{n};u^{n},v)-\tilde{a}(Pu^{n};Pu^{n},v)\\
		\leq & C\Big(\|u^{n}-P_{H}u^{n}\|_{\kappa}+\|\nabla u\|_{L^{\infty}([0,T],L^{\infty})}\|\kappa_{x}^{\frac{1}{2}}(u^{n}-P_{H}u^{n})\|_{L^{2}}\Big)\|v\|_{\kappa}\\
		\leq & C(1+H\|\nabla u\|_{L^{\infty}([0,T],L^{\infty})})\|u^{n}-P_{H}u^{n}\|_{\kappa}\|v\|_{\kappa}.
	\end{align*}
	
	Finally, we estimate $\|r_{\delta a,i}^{n}\|_{\kappa^{*}}$ and have
	\begin{align*}
		(r_{\delta a,i}^{n},v) & =\Delta ta_{0}(P_{H.i}(\cfrac{u^{n+1}-u^{n}}{\Delta t}),v) =\Delta ta_{0}(P_{H.i}(\partial_{t}u(\xi)),v)\\
		& \leq\Delta t\|P_{H.i}(\partial_{t}u(\xi))\|_{\kappa}\|v\|_{\kappa} \leq C\Delta t\|\partial_{t}u(\xi)\|_{\kappa}\|v\|_{\kappa}.
	\end{align*}
	Therefore, we have 
	\begin{align*}
		\|r_{a}^{n}\|_{\kappa^{*}} & \leq C(1+H\|\nabla u\|_{L^{\infty}([0,T],L^{\infty})})\|u^{n}-P_{H}u^{n}\|_{\kappa}\\
		\|r_{\delta a,i}^{n}\|_{\kappa^{*}} & \leq C\Delta t\|\partial_{t}u(\xi)\|_{\kappa}.
	\end{align*}
 We next estimate the term $\|\cfrac{\delta r_{\delta a,i}^{n}}{\Delta t}\|_{\kappa^{*}}$
and $\|\cfrac{\delta r_{a}^{n}}{\Delta t}\|_{\kappa^{*}}$. We first
have

\begin{align*}
(\delta r_{\delta a,i}^{n},v) & =\Delta t^{2}a_{0}(P_{H.i}(\cfrac{\delta u^{n+1}-\delta u^{n}}{\Delta t^{2}}),v)
\leq C\Delta t^{2}\|\partial_{tt}u\|_{\kappa}(\xi)\|v\|_{\kappa}.
\end{align*}

By definition of $(\delta r_{a}^{n},v)$, we have 
\begin{align*}
(\delta r_{a}^{n},v)= & \tilde{a}(u^{n+1};u^{n+1},v)-\tilde{a}(u^{n};u^{n},v)\\
 & -\tilde{a}(P_{H}u^{n+1};P_{H}u^{n+1},v)+\tilde{a}(P_{H}u^{n};P_{H}u^{n},v).
\end{align*}
We then separate $(\delta r_{a}^{n},v)$ into four terms such that

\[
(\delta r_{a}^{n},v)=J_{1}+J_{2}+J_{3}+J_{4},
\]
where 
\begin{align*}
J_{1} & =\int_{\Omega}\kappa_{x}(I-P_{H})\Big(\kappa(u^{n+1})-\kappa(u^{n})\Big)\nabla u^{n+1}\cdot\nabla v,\\
J_{2} & =\int_{\Omega}\kappa_{x}P_{H}\Big(\kappa(u^{n+1})-\kappa(u^{n})\Big)\nabla(I-P_{H})u^{n+1}\cdot\nabla v,\\
J_{3} & =\int_{\Omega}(\kappa_{x}\kappa(P_{H}u^{n})-\kappa_{0})\nabla((I-P_{H})\delta u^{n+1})\cdot\nabla v,\\
J_{4} & =\int_{\Omega}(\kappa_{x}(I-P_{H})\kappa(u^{n}))\nabla(\delta u^{n+1})\cdot\nabla v.
\end{align*}

We then have 
\begin{align*}
J_{1} & \leq C\Delta t\|\kappa_{x}^{\frac{1}{2}}\nabla u\|_{L^{\infty}([0,T],L^{\infty})}\|(I-P_{H})\partial_{t}\kappa(u)\|_{L^{\infty}([0,T],L^{2})}\|v\|_{\kappa},\\
J_{2} & \leq C\Delta t\|P_{H}\Big(\partial_{t}\kappa(u)\Big)\|_{L^{\infty}([0,T],L^{\infty})}\|(I-P_{H})u\|_{L^{\infty}([0,T],\kappa)}\|v\|_{\kappa},\\
J_{3} & \leq C\Delta t\|(I-P_{H})\partial_{t}u\|_{L^{\infty}([0,T],\kappa)}\|v\|_{\kappa},\\
J_{4} & \leq C\Delta t\|\kappa_{x}^{\frac{1}{2}}\nabla u_{t}\|_{L^{\infty}([0,T],L^{\infty})}\|(I-P_{H})\kappa(u)\|_{L^{\infty}([0,T],\kappa)}\|v\|_{\kappa}.
\end{align*}

\end{proof}

\section{Numerical examples} \label{sec:numerical}
In this section, we will present some numerical tests and demonstrate the performance of the proposed algorithm.
The problem is posed on a square domain $\Omega$ and the simulation time interval is $[0,T]$. 

\subsection{Example 1}\label{subsec:num_ex1}
In the first example, we solve the following nonlinear equation 
\begin{equation*}
	\begin{aligned}
		u_t - \nabla \cdot (\kappa_0 \exp(\beta u)  \nabla u) &= f \;\;\;\; \text{ on } \Omega\times (0,T] \\
		\nabla u\cdot \boldsymbol{n} &= 0 \;\;\;\; \text{ on }  \partial \Omega  \times (0,T]
	\end{aligned}
\end{equation*}
with initial condition $u(0)= 0$. In the above equation, $\Omega = [0,1]\times[0,1]$, $T=0.1$, and $\beta=1$. The source term $f=1$ at $(x,y)=(0.31,0.11)$ and $f=0$ elsewhere in the domain. $\kappa_0$ is a channelized permeability field, The value of permeability is $10^4$ inside channels, and is $1$ outside the channels. The configuration of the permeability field is shown in Figure \ref{fig:source_perm_ex1}. 

The time step size is $\Delta t =5\times 10^{-5}$. To compute reference solutions, we use piecewise linear finite elements on a $100\times 100$ grid for the spatial discretization. Newton iterations are performed at each time step, and the iteration stops when the norm of the update is less than $10^{-10}$. For coarse level approximations, the coarse grid is $10\times 10$, and we compute NLMC and enriched NLMC basis (ENLMC) with respect to $\kappa_0$ as discussed in Section \ref{sec:spaces}. The degrees of freedom of fine grid approximation is $10201$, and the degrees of freedom of coarse grid approximation is listed in the table for NLMC and ENLMC.

We then compute the following approximations of the solution: (1) fine grid approximation with implicit Euler \eqref{eq:fine_implicit}-\eqref{eq:NR}, (2) coarse scale approximation with NLMC basis and implicit Euler, (3) coarse scale approximation with ENLMC basis and implicit Euler, (4) partially explicit scheme \eqref{eq:partial_exp} with EIN, which is our proposed method. We compare the errors and running time of (2)-(4) with the reference (1). The errors are presented in Figure \ref{tab:ex1_err1}. From Figure \ref{fig:errs_ex1}, we observe that our proposed method can give similar results as (2) and (3). However, our approach takes less computational time, see Table \ref{tab:ex1_time}.

We also show the profiles of the solution obtained from our proposed approach and the reference solution at the second time step and final times step in Figure \ref{fig:sol_ex1}. 

\begin{figure}
	\centering
	\includegraphics[width=0.8\textwidth]{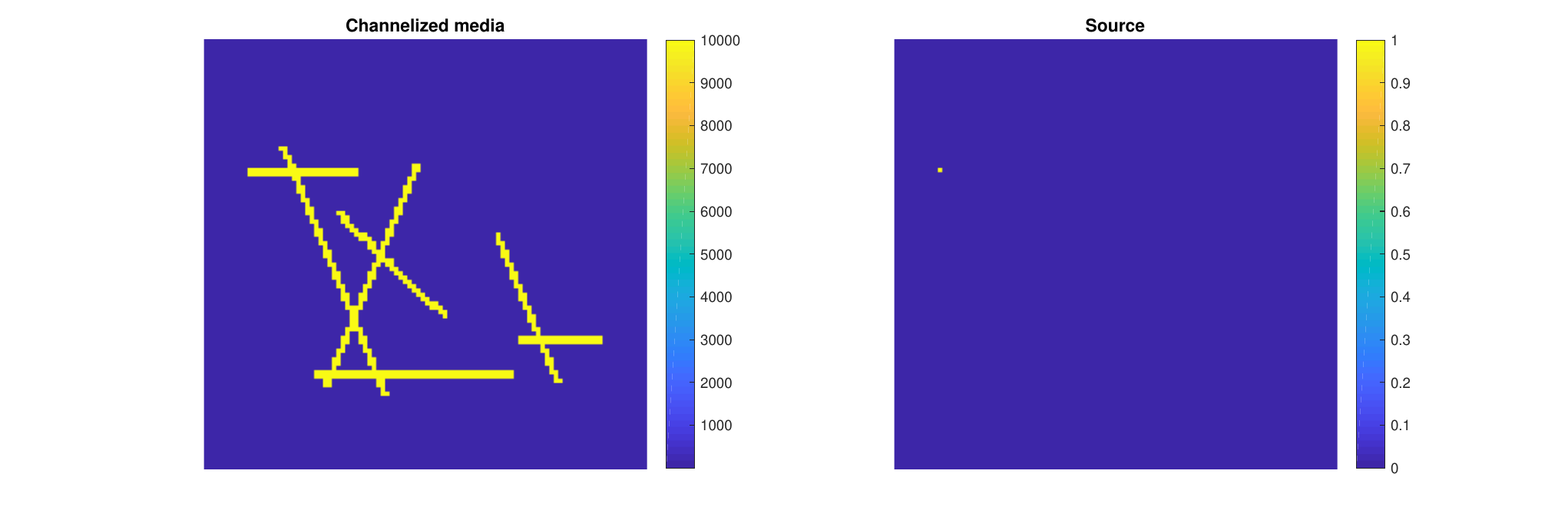}
	\caption{Example 1, left: the channelized permeability field, right: the source term.}
	\label{fig:source_perm_ex1}
\end{figure}

\begin{figure}
	\centering
	\includegraphics[width=0.8\textwidth]{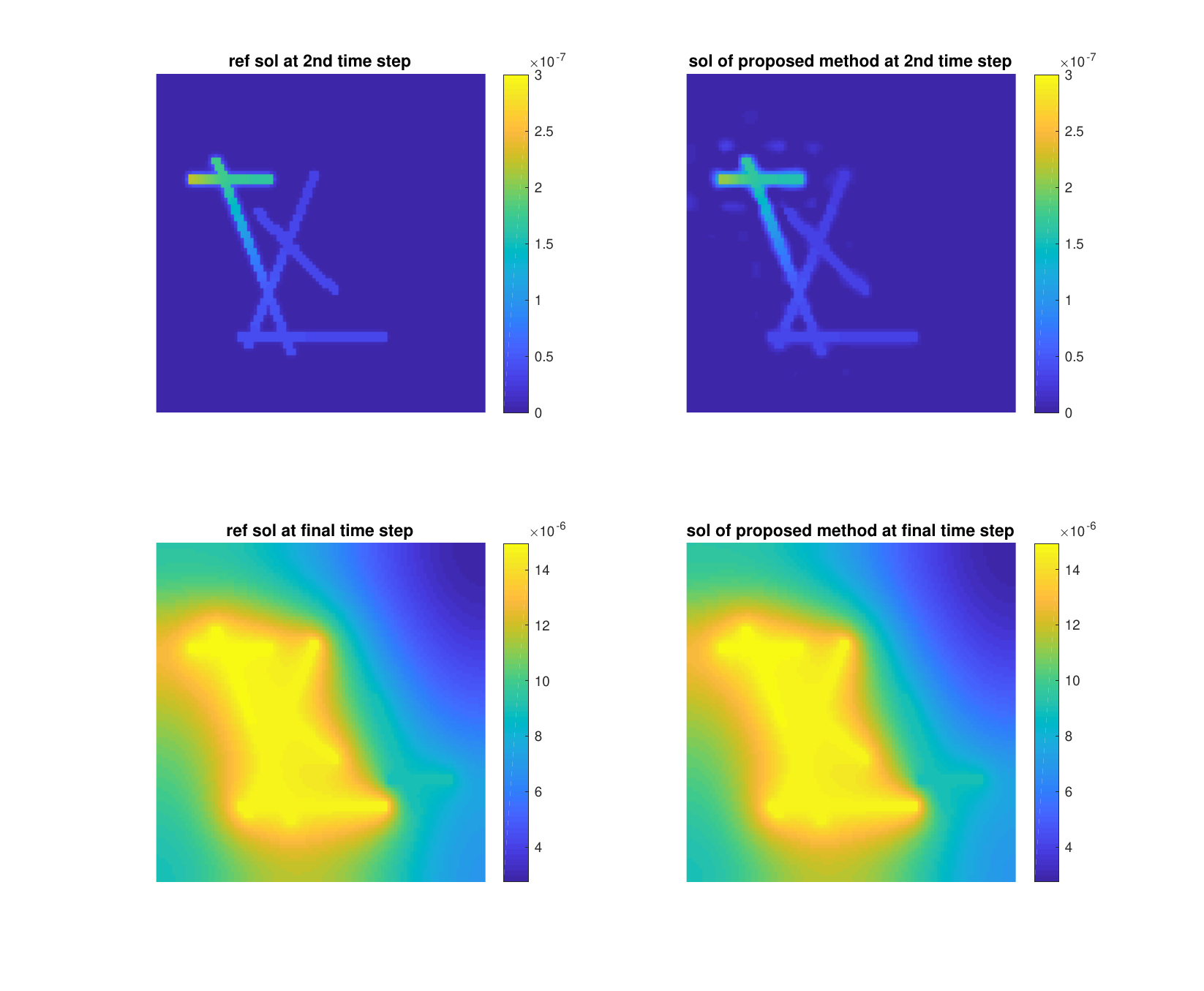}
	\caption{Example 1, the comparison of solutions at different time steps. Left: reference solutions, right: solutions obtained from the proposed algorithm with NLMC basis.}
	\label{fig:sol_ex1}
\end{figure}

\begin{figure}
	\centering
	\includegraphics[width=0.45\textwidth]{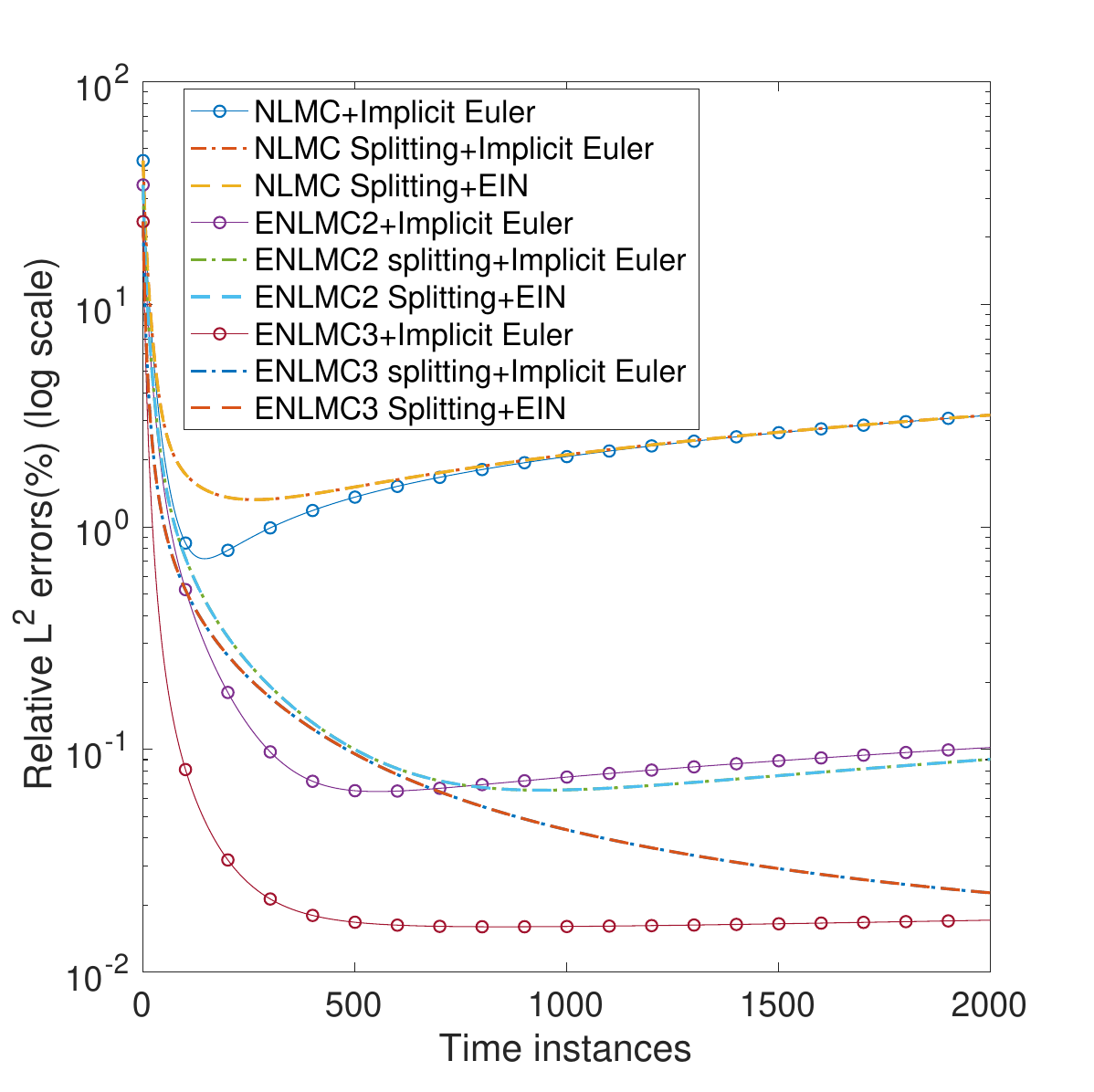}
	\includegraphics[width=0.45\textwidth]{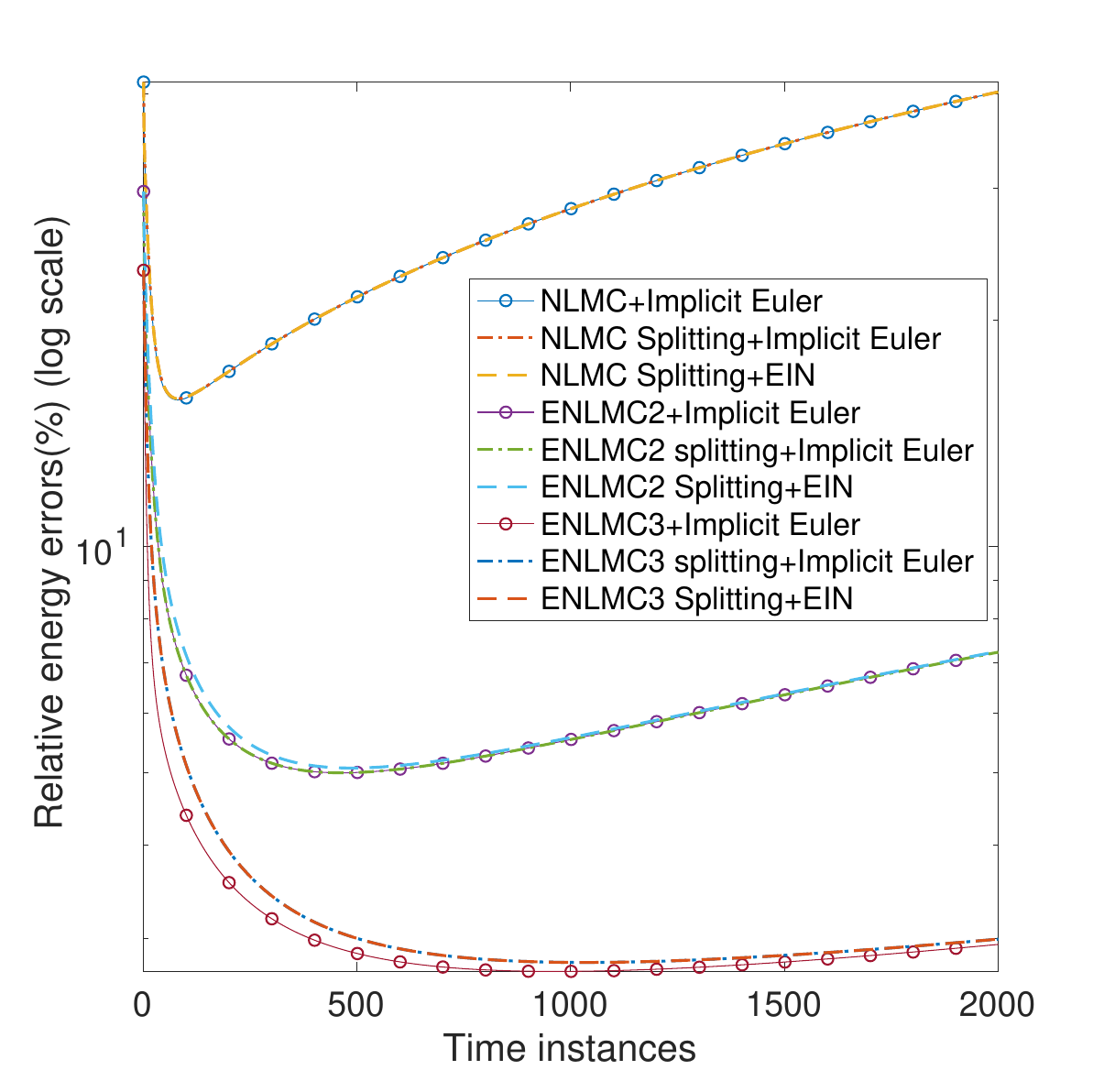}
	\caption{Example 1, error history between the reference solution and the approximation solutions.}
	\label{fig:errs_ex1} 
\end{figure}

\begin{table}[!htb]
	\centering
	\begin{tabular}{|c|c|c|c|}\hline
		no. of basis($V_{H,1}/V_{H,2}$) &DOF($V_{H,1}/V_{H,2}$)&\makebox[5em]{$L_2$ error(\%)}&\makebox[5em]{Energy error(\%)}  \\\hline
		NLMC 1/1 &26/100 &2.38 &28.01   \\\hline  
		ENLMC 2/2  &78/300  &0.31 &6.17\\\hline
		ENLMC 3/3  &104/400  &0.19  &3.27 \\\hline
	\end{tabular}
	\caption{Example 1, relative $L_2$ and energy errors (average over all time steps) between the reference solutions and the solutions from the EIN splitting scheme.  }
	\label{tab:ex1_err1}
\end{table}

\begin{table}[!htb]
	\centering
	\begin{tabular}{|c|c|c|c|}\hline
		&No splitting+IE &Splitting+IE &Splitting+EIN  \\\hline
		NLMC 1+1   &11.50 &5.56 &3.94   \\\hline  
		ENLMC 2+2  &20.48 &18.15 &8.08 \\\hline
		ENLMC 3+3  &34.78 &31.67  &15.66 \\\hline	
	\end{tabular}
	\caption{Example 1, the running time (in seconds) for different cases (IE refers to implicit Euler).}
	\label{tab:ex1_time}
\end{table}

\subsection{Example 2}
In the second example, we consider a compressible fluid flow problem
\begin{equation*}
	\begin{aligned}
		\partial_t(\phi \rho(u)) - \nabla \cdot (\frac{\kappa}{\mu} \rho(u)  \nabla u) &= f \;\;\;\; \text{ on } \Omega\times (0,T] \\
		\nabla u\cdot \boldsymbol{n} &= 0 \;\;\;\; \text{ on }  \partial \Omega  \times (0,T]\\
		u &= u_0 \;\;\;\; \text{ on }   \Omega  \times \{t=0\}
	\end{aligned}
\end{equation*}
where the computational domain $\Omega = [0,64m]\times [0,64m]$, $T=25.2$ minutes. The porosity $\phi = 500$, viscosity $\mu = 5cP$, and compressibility parameter $c = 10^{-8} Pa^{-1}$. The initial pressure $u_0 = 2.16 \times 10^7 Pa$. We take zero Neumann boundary condition on the whole boundary of the computational domain.

The fluid density 
\begin{equation*}
	\rho(u)  = \rho_{\text{ref}} e^{c(u-u_{\text{ref}})},
\end{equation*}
where $\rho_{\text{ref}} = 850 kg/m^3$ is the reference density, and $u_{\text{ref}} = 2\times 10^{7} Pa$ is the reference pressure.

The time step size is $\Delta t = 0.6048$ seconds. As in the previous example, the reference solutions is computed on a $100\times 100$ grid, using implicit Euler with Newton iterations. The coarse grid is $10\times 10$, and the degrees of freedom of coarse grid approximation are listed in Tables \ref{tab:ex2_err1} for NLMC and ENLMC approaches. The absolute permeability $\kappa$ in this example is shown in Figure \ref{fig:source_perm_ex2}, the values in the yellow region (channels) are $10^3$, and the values in the blue region is $1$.

As before, we compare the four approximation solutions described in section \ref{subsec:num_ex1} with the reference solution. From Figure \ref{fig:errs_ex2}, we observe that our proposed method can give similar results as the case for no splitting NLMC/ENLMC multiscale method with standard implicit Euler for time discretization. However, our approach takes less computational time. The errors and running time are presented in Table \ref{tab:ex2_err1} and Table \ref{tab:ex2_time}.Our MATLAB code is run on a macOS, and the recorded time shows that with the proposed EIN splitting approach, the computational cost can be saved significantly.

We also show the profiles of the solution obtained from our proposed approach with 4 ENLMC basis in each local region (2 for $V_{H,1}$ and 2 for $V_{H,2}$), and the reference solutions, at the second time step and final times step in Figure \ref{fig:sol_ex2}. The figure shows that our approach gives very good approximation results.

\begin{figure}
\centering
\includegraphics[width=0.8\textwidth]{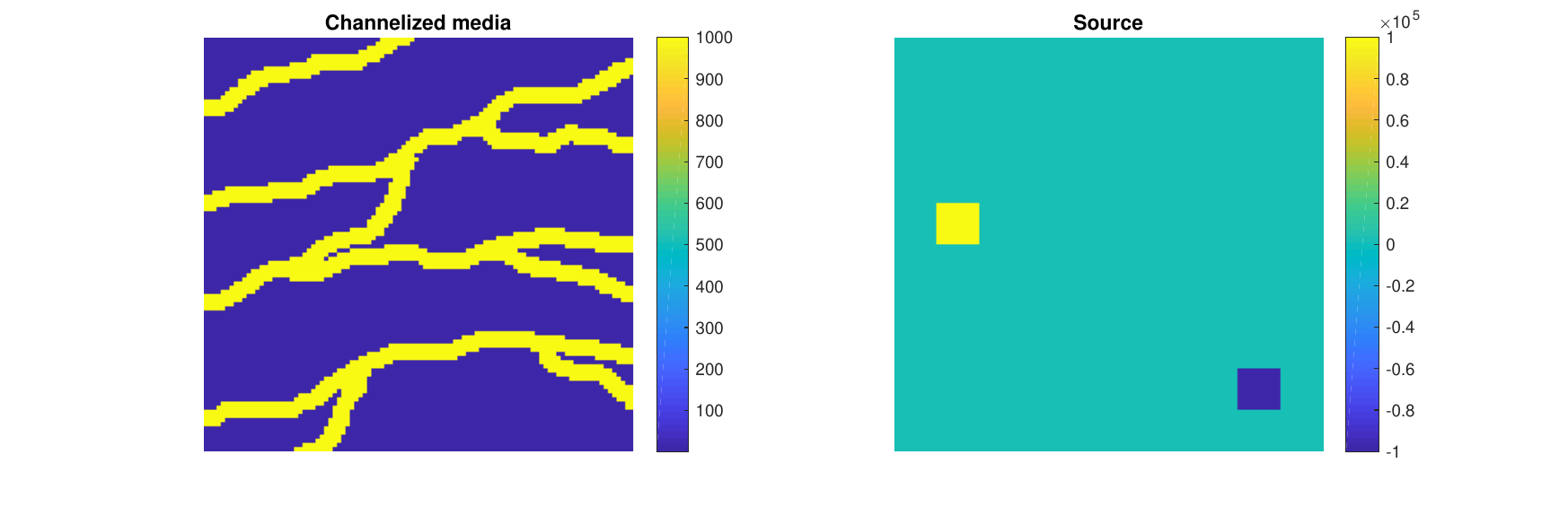}
\caption{Example 2, left: the channelized permeability field, right: the source term.}
\label{fig:source_perm_ex2}
\end{figure}

\begin{figure}
\centering
\includegraphics[width=0.8\textwidth]{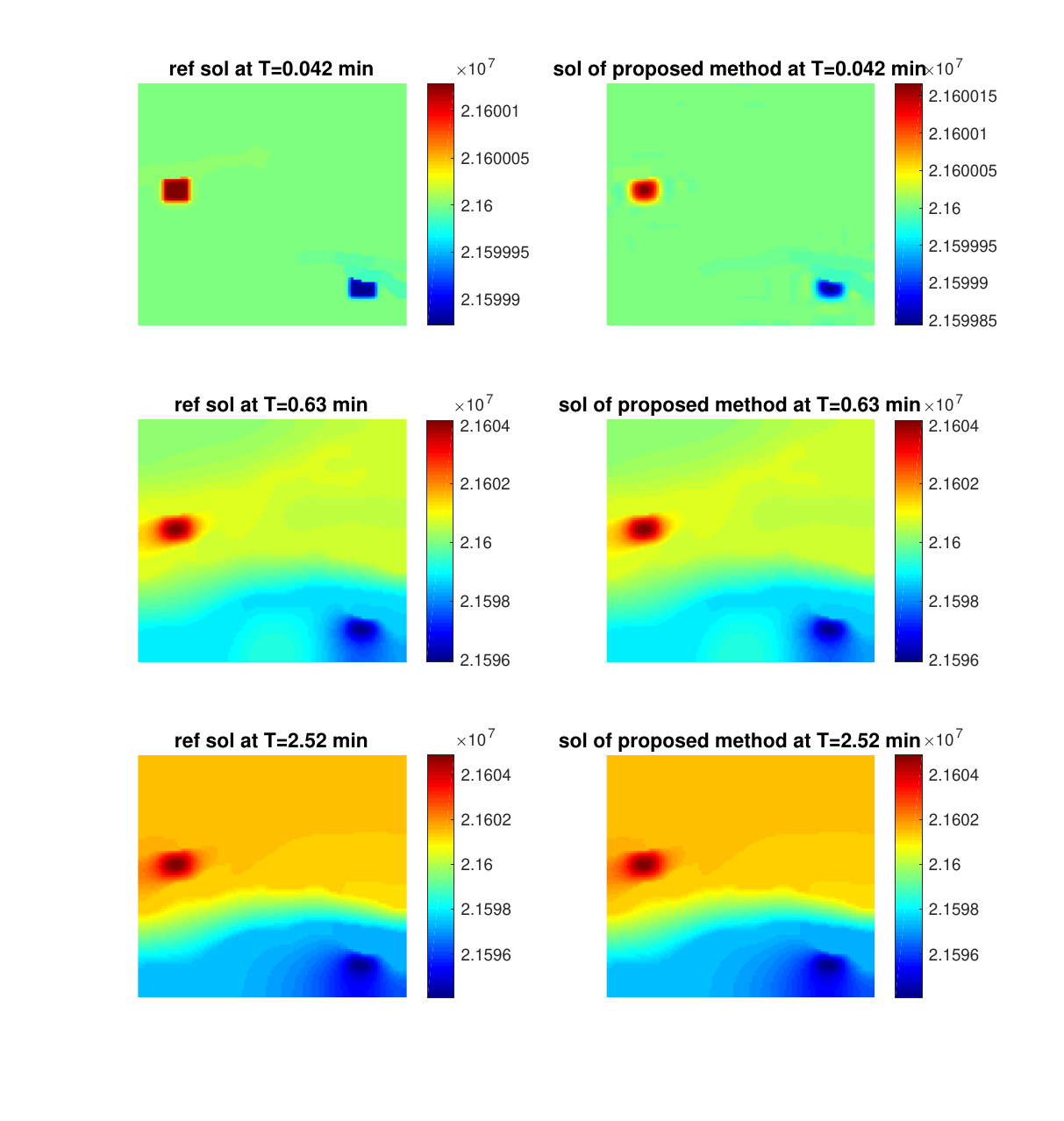}
\caption{Example 2, the comparison of solutions at different time steps. Left: reference solutions, right: solutions obtained from the proposed algorithm.}
\label{fig:sol_ex2}
\end{figure}

\begin{figure}
\centering
\includegraphics[width=0.45\textwidth]{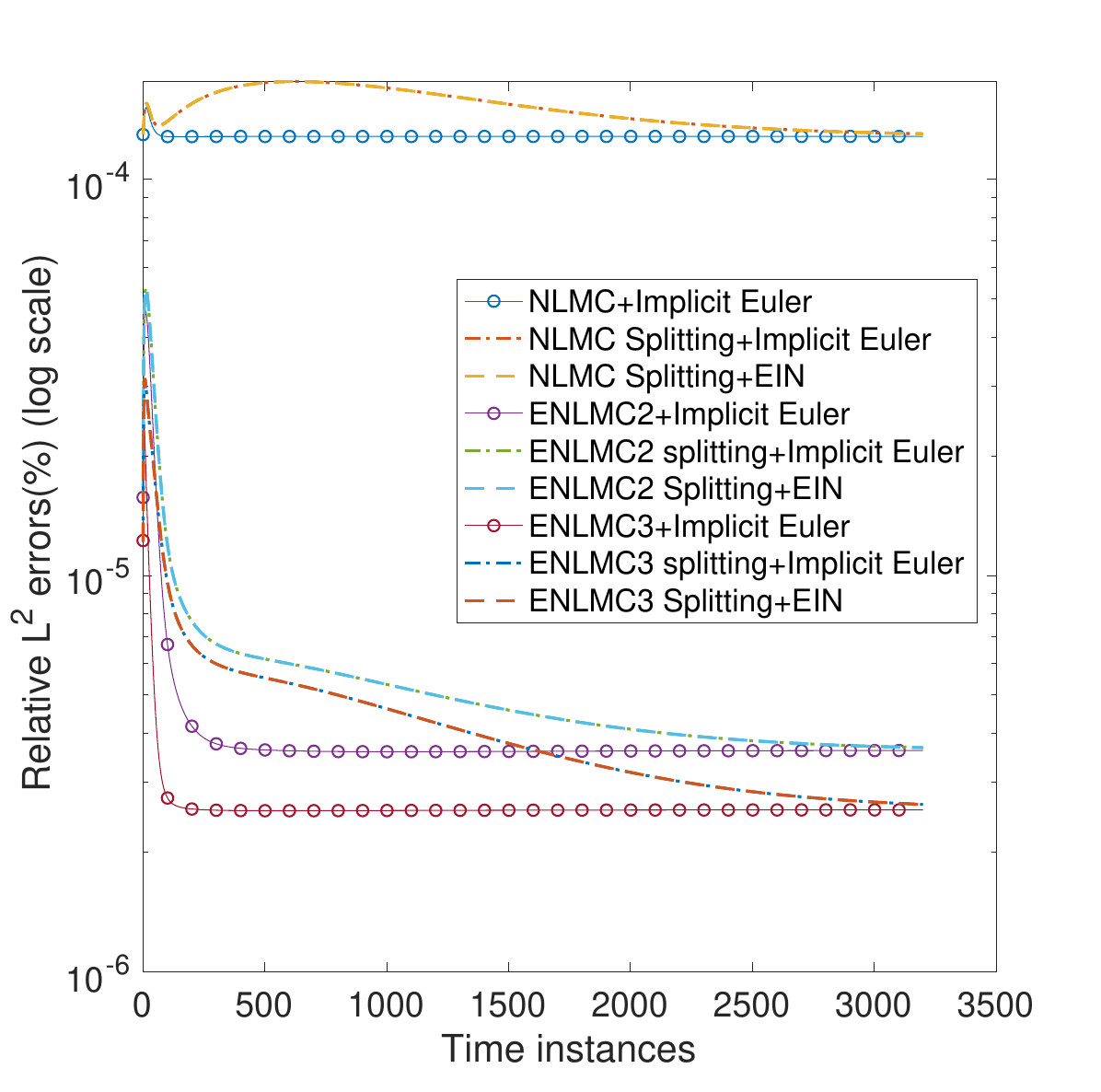}
\includegraphics[width=0.45\textwidth]{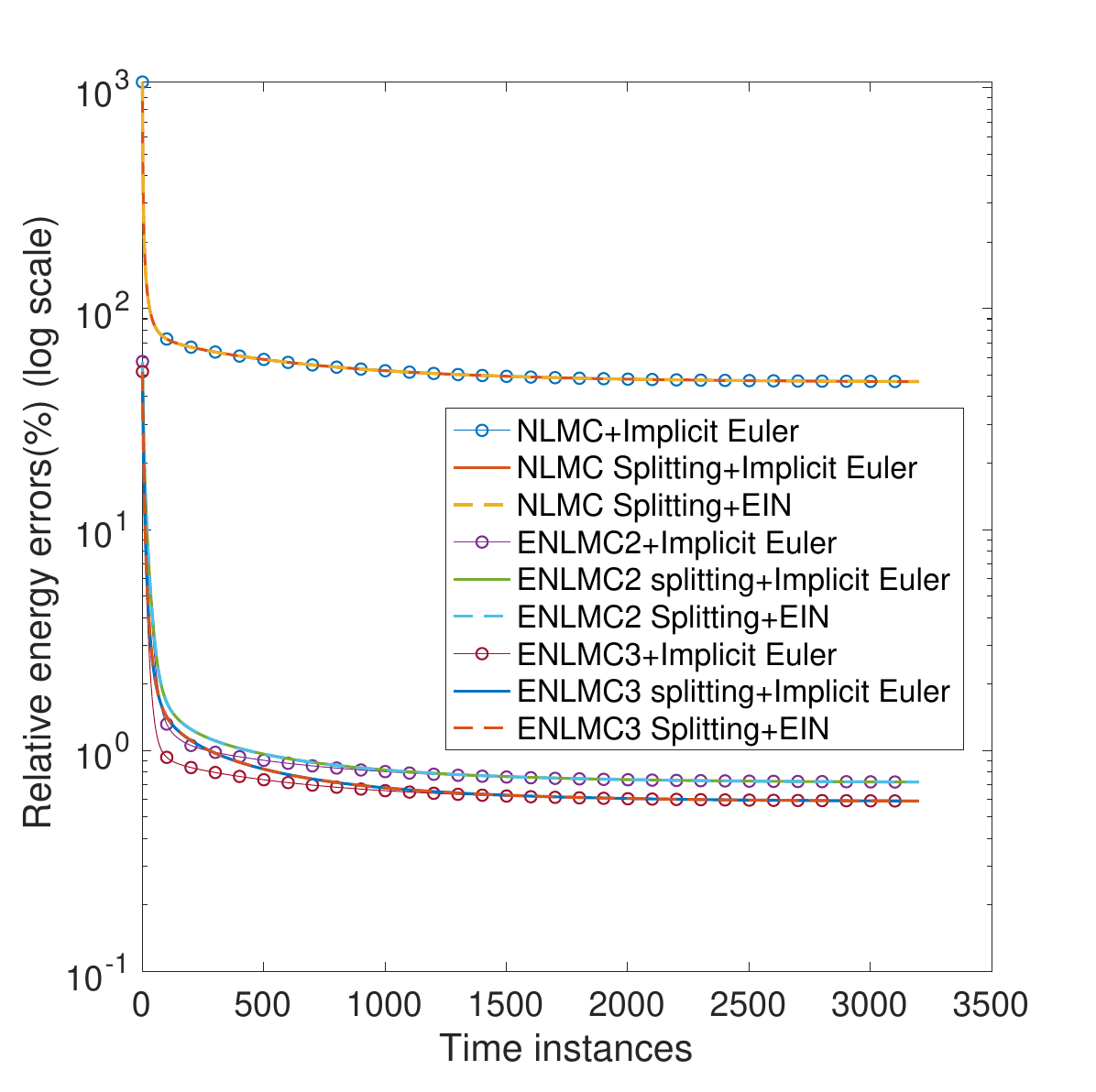}
\caption{Example 2, error history between the reference solution and the approximation solutions.}
\label{fig:errs_ex2} 
\end{figure}

\begin{table}[!htb]
\centering
\begin{tabular}{|c|c|c|c|c|}\hline
	no. of basis ($V_{H,1}/V_{H,2}$) &DOF($V_{H,1}/V_{H,2}$)&\makebox[5em]{$L_2$ error(\%)}&\makebox[5em]{Energy error(\%)}  \\\hline
	NLMC 1/1 &60/100 &1.56 E{-4} &53.59   \\\hline  
	ENLMC 2/2  &120/200  &5.59 E{-6}  & 1.02\\\hline
	ENLMC 3/3  &180/300  &4.41 E{-6} &0.82 \\\hline
\end{tabular}
\caption{Example 2, relative $L_2$ and energy errors (average over all time steps) between the reference solutions and the solutions from the EIN splitting scheme.  }
\label{tab:ex2_err1}
\end{table}

\begin{table}[!htb]
\centering
\begin{tabular}{|c|c|c|c|}\hline
	&No splitting+IE&  Splitting+IE &Splitting+EIN  \\\hline
	NLMC 1+1  & 270.44 & 254.80  &7.51   \\\hline  
	ENLMC 2+2   &509.26 &466.09 &20.09 \\\hline
	ENLMC 3+3   &967.02 &727.37 &35.85 \\\hline
\end{tabular}
\caption{Example 2, the running time (in seconds) for different cases. The running time for the reference solution is 13216.623606 seconds (IE refers to implicit Euler).}
\label{tab:ex2_time}
\end{table}

\section{Conclusion} \label{sec:conclusion}
In this work, we proposed a partially explicit splitting scheme with explicit-implicit-null method for nonlinear multiscale diffusion problems. The explicit-implicit-null approach is to add and substract a linear diffusion term in the equation. However, compared to other works, the damping term we introduced is not a constant, but a dominant multiscale coefficient in order to capture the underlying property. We then proposed the partially explicit splitting scheme for the linear multiscale part of the new equation. This further improves the efficiency of temporal discretization. The stability of the proposed scheme is presented. The efficiency of the proposed scheme are shown in the application of Richard's equation and compressible multiscale flow problems.


\bibliographystyle{plain} 
\bibliography{references}

\appendix
\section{Proof of Lemma \ref{lem:error_a} and \ref{lem:error_L2}} 

In this section, we will discuss the proof of Lemma \ref{lem:error_a} and \ref{lem:error_L2} in detail. First of all, we will first introduce the proof for Lemma \ref{lem:error_a}. 
We will first estimate the consistence of the temporal discretization.
	Given $u$ is the solution of (\ref{eq:weak_eq}), we have
	\begin{align*}
		(\partial_{t}u^{n},v)+a(u^{n};u^{n},v) & =(\partial_{t}u^{n},v)+a_{0}(u^{n},v)+\tilde{a}(u^{n};u^{n},v)
		=(f^{n},v)\;\forall v\in V.
	\end{align*}
	We then estimate the  discretization error of term $\partial_t u$, and consider
	\begin{align*}
		\cfrac{P_{H,1}(\delta u^{n+1})+P_{H,2}(\delta u^{n})}{\Delta t} & =\cfrac{P_{H}(\delta u^{n+1})+P_{H,2}(-u^{n+1}+2u^{n}-u^{n-1})}{\Delta t}\\
		& =\partial_{t}u^{n}-r_{\delta}^{n}+r_{\delta^{2},2}^{n}.
	\end{align*}
	Similarly, we have 
		$\cfrac{P_{H,2}(\delta u^{n+1})+P_{H,1}(\delta u^{n})}{\Delta t} =\partial_{t}u^{n}-r_{\delta}^{n}+r_{\delta^{2},1}^{n}.$
	
	Thus, we obtain
	\begin{align*}
		& (\cfrac{P_{H,i}(\delta u^{n+1})+P_{H,j}(\delta u^{n})}{\Delta t}+r_{\delta}^{n}-r_{\delta^{2},j}^{n},v)\\
		& +a_{0}(P_{H,1}(u^{n+1})+P_{H,2}(u^{n}),v)+(r_{\delta a,2}^{n},v) +\tilde{a}(Pu^{n};Pu^{n},v)+(r_{a}^{n},v)\\
		= & (\partial_{t}u^{n},v)+a_{0}(u^{n},v)+\tilde{a}(u^{n};u^{n},v)=(f^{n},v)\;\forall v\in V_{H,i},\;i=1,2,\;j=\text{mod}(i,2)+1.
	\end{align*}
	For the error terms $\eta_{i}^{n}=P_{H,i}u^{n}-u_{H,i}^{n}$ 
	and $\eta^{n}=\sum_{i=1,2}\eta_{i}^{n}$, we have
	\begin{align*}
		&\tilde{a}(\eta^{n};\eta^{n},v)= a(\eta^{n};\eta^{n},v)-a_{0}(\eta^{n},v)\\
		= & a(\eta^{n};\eta^{n},v)-a(P_{H}u^{n};\eta^{n},v)-a(P_{H}u^{n};u_{H}^{n},v)+a(u_{H}^{n};u_{H}^{n},v)\\
		& +\tilde{a}(P_{H}u^{n};P_{H}u^{n},v)-\tilde{a}(u_{H}^{n};u_{H}^{n},v).
	\end{align*}
	Therefore, we have 
	\begin{align}
		(\cfrac{\delta\eta_{1}^{n+1}+\delta\eta_{2}^{n}}{\Delta t},v)+a_{0}(\eta_{1}^{n+1}+\eta_{2}^{n},v)+b(P_{H}u^{n},u_H^n,v) & =(r_{1}^{n},v)\;\forall v\in V_{H,1}, \label{eq:residual_1} \\ 
		(\cfrac{\delta\eta_{2}^{n+1}+\delta\eta_{1}^{n}}{\Delta t},v)+a_{0}(\eta_{1}^{n+1}+\eta_{2}^{n},v)+ b(P_{H}u^{n},u_H^n,v)& =(r_{2}^{n},v)\;\forall v\in V_{H,2},\label{eq:residual_2}
	\end{align}
	where $r_{1}^{n}:=-r_{\delta}^{n}+r_{\delta^{2},2}^{n}-r_{\delta a,2}^{n}-r_{a}^{n}$,
	$r_{2}^{n}:=-r_{\delta}^{n}+r_{\delta^{2},1}^{n}-r_{\delta a,2}^{n}-r_{a}^{n}$ and 
	\[
	b(P_{H}u^{n},u_H^n,v) := \tilde{a}(P_{H}u^{n};P_{H}u^{n},v)-\tilde{a}(u_{H}^{n};u_{H}^{n},v)
	\]
	
	Putting $v=\delta\eta_{1}^{n+1}$ in (\ref{eq:residual_1}) and $v=\delta\eta_{2}^{n+1}$ in (\ref{eq:residual_2}) and summing up these two equations, we obtain 
	\begin{align}
		(\cfrac{\delta\eta_{1}^{n+1}+\delta\eta_{2}^{n}}{\Delta t},\delta\eta_{1}^{n+1})+(\cfrac{\delta\eta_{2}^{n+1}+\delta\eta_{1}^{n}}{\Delta t},\delta\eta_{2}^{n+1}) \label{eq:delta_sum_eq_1}\\ 
		+a_{0}(\eta_{1}^{n+1}+\eta_{2}^{n},\delta\eta^{n+1})+b(P_{H}u^{n},u_H^n,\delta\eta^{n+1}) & =(r_{1}^{n},\delta\eta_{1}^{n+1})+(r_{2}^{n},\delta\eta_{2}^{n+1}).\label{eq:delta_sum_eq_2}
	\end{align}	
	We estimate the terms in (\ref{eq:delta_sum_eq_1}) using the same calculation as the proof of Lemma \ref{lemma1}
	\begin{align*}
		& (\cfrac{\delta\eta_{1}^{n+1}+\delta\eta_{2}^{n}}{\Delta t},\delta\eta_{1}^{n+1})+(\cfrac{\delta\eta_{2}^{n+1}+\delta\eta_{1}^{n}}{\Delta t},\delta\eta_{2}^{n+1}) \\
		\geq &\frac{1}{\Delta t} \sum_{i=1,2}\Big(\Big(1-\cfrac{\gamma}{2}\Big)\|\delta\eta_{i}^{n+1}\|^{2}-\cfrac{\gamma}{2}\|\delta\eta_{i}^{n}\|^{2}\Big).
	\end{align*}
	For the last two terms in the left hand side of 
	(\ref{eq:delta_sum_eq_2}), we have
	\begin{align*}
		& a_{0}(\eta_{1}^{n+1}+\eta_{2}^{n},\delta\eta^{n+1})+\tilde{a}(P_{H}u^{n};P_{H}u^{n},\delta\eta^{n+1})-\tilde{a}(u_{H}^{n};u_{H}^{n},\delta\eta^{n+1})\\
		= & a_{0}(\eta_{1}^{n+1}+\eta_{2}^{n},\delta\eta^{n+1})-a_{0}(\eta^{n},\delta\eta^{n+1})+a(P_{H}u^{n};P_{H}u^{n},\delta\eta^{n+1})-a(u_{H}^{n};u_{H}^{n},\delta\eta^{n+1})\\
		= & a_{0}(\eta^{n+1},\delta\eta^{n+1})-a_{0}(\delta\eta_{2}^{n+1},\delta\eta^{n+1})+ A_1,
	\end{align*}
	where $A_1 = a(P_{H}u^{n};P_{H}u^{n},\delta\eta^{n+1})-a(u_{H}^{n};u_{H}^{n},\delta\eta^{n+1})-a_{0}(\eta^{n},\delta\eta^{n+1})$.
	
	We then estimate the term $A_1$,
	and obtain
	\begin{align*}
		& a(P_{H}u^{n};P_{H}u^{n},\delta\eta^{n+1})-a(u_{H}^{n};u_{H}^{n},\delta\eta^{n+1})-a_{0}(\eta^{n},\delta\eta^{n+1})\\
		= & \int_{\Omega}\kappa_{x}\Big(\kappa_{u}(P_{H}u^{n})-\kappa_{u}(u^{n})\Big)\nabla u_{H}^{n}\cdot\nabla\delta\eta^{n+1}+\int_{\Omega}\Big(\kappa_{x}\kappa_{u}(P_{H}u^{n})-\kappa_{0}\Big)\nabla\eta^{n}\cdot\nabla\delta\eta^{n+1}\\
		= & I_{1}+\int_{\Omega}(\cfrac{\kappa_{x}\kappa_{u}(P_{H}u^{n})-\kappa_{0}}{2})\Big(|\nabla\eta^{n+1}|^{2}-|\nabla\eta^{n}|^{2}-|\nabla\delta\eta^{n+1}|^{2}\Big)
	\end{align*}
	where $I_{1}:=\int_{\Omega}\kappa_{x}\Big(\kappa_{u}(P_{H}u^{n})-\kappa_{u}(u^{n})\Big)\nabla u_{H}^{n}\cdot\nabla\delta\eta^{n+1}.$
	
	Since $a_{0}(\eta^{n+1},\delta\eta^{n+1})=\int_{\Omega}\cfrac{\kappa_{0}}{2}\Big(|\nabla\eta^{n+1}|^{2}-|\nabla\eta^{n}|^{2}+|\nabla\delta\eta^{n+1}|^{2}\Big)$,
	we have 
	\begin{align*}
		& a_{0}(\eta^{n+1},\delta\eta^{n+1})-a_{0}(\delta\eta_{2}^{n+1},\delta\eta^{n+1})+A_1\\
		= & I_{1}+\int_{\Omega}(\cfrac{\kappa_{x}\kappa_{u}(P_{H}u^{n})}{2})\Big(|\nabla\eta^{n+1}|^{2}-|\nabla\eta^{n}|^{2}\Big)\\
		&+ \int_{\Omega}(\cfrac{2\kappa_{0}-\kappa_{x}\kappa_{u}(P_{H}u^{n})}{2})|\nabla\delta\eta^{n+1}|^{2}-\int_{\Omega}\Big(\kappa_{0}\nabla\delta\eta_{2}^{n+1}\Big)\cdot\nabla\delta\eta^{n+1}
	\end{align*}
	
	Since $|\cfrac{\kappa_{0}-\kappa_{x}\kappa_{u}(P_{H}u^{n})}{\kappa_{0}}|\leq C_{1}$,
	we have $\cfrac{2\kappa_{0}-\kappa_{x}\kappa_{u}(P_{H}u^{n})}{2}\geq\cfrac{1-C_{1}}{2}\kappa_{0}$
	and 
	\begin{align*}
		& \int_{\Omega}(\cfrac{2\kappa_{0}-\kappa_{x}\kappa_{u}(P_{H}u^{n})}{2})|\nabla\delta\eta^{n+1}|^{2}-\int_{\Omega}\Big(\kappa_{0}\nabla\delta\eta_{2}^{n+1}\Big)\cdot\nabla\delta\eta^{n+1}\\
		\geq & \cfrac{-1}{(1-C_{1})}\int_{\Omega}\kappa_{0}|\nabla\delta\eta_{2}^{n+1}|^{2}+(\cfrac{1-C_{1}}{4})\int_{\Omega}\kappa_{0}|\nabla\delta\eta^{n+1}|^{2}.
	\end{align*}
	Combining the equations, we obtain
	\begin{align*}
		& \Delta t^{-1}\sum_{i=1,2}\Big(\Big(1-\gamma+\cfrac{\gamma}{2}\Big)\|\delta\eta_{i}^{n+1}\|^{2}-\cfrac{\gamma}{2}\|\delta\eta_{i}^{n}\|^{2}\Big)-\cfrac{1}{(1-C_{1})}\int_{\Omega}\kappa_{0}|\nabla\delta\eta_{2}^{n+1}|^{2}\\
		& +(\cfrac{1-C_{1}}{4})\|\delta\eta^{n+1}\|_{\kappa}^{2}+I_{1}+\int_{\Omega}(\cfrac{\kappa_{x}\kappa_{u}(P_{H}u^{n})}{2})\Big(|\nabla\eta^{n+1}|^{2}-|\nabla\eta^{n}|^{2}\Big)\\
		\leq & (\cfrac{\delta\eta_{1}^{n+1}+\delta\eta_{2}^{n}}{\Delta t},\delta\eta_{1}^{n+1})+(\cfrac{\delta\eta_{2}^{n+1}+\delta\eta_{1}^{n}}{\Delta t},\delta\eta_{2}^{n+1})+I_{1}\\
		& +\int_{\Omega}(\cfrac{\kappa_{x}\kappa_{u}(P_{H}u^{n})}{2})\Big(|\nabla\eta^{n+1}|^{2}-|\nabla\eta^{n}|^{2}\Big)+\int_{\Omega}(\cfrac{2\kappa_{0}-\kappa_{x}\kappa_{u}(P_{H}u^{n})}{2})|\nabla\delta\eta^{n+1}|^{2}\\
		\leq & (r_{1}^{n},\delta\eta_{1}^{n+1})+(r_{2}^{n},\delta\eta_{2}^{n+1}).
	\end{align*}
	Since $\sup_{v_{2}\in V_{H,2}\backslash\{0\}}\cfrac{\|v_{2}\|_{\kappa}^{2}}{\|v_{2}\|^{2}}\leq\cfrac{(1-\gamma)(1-C_{1})}{2\Delta t}$,
	we have 
	\begin{align*}
		& \Delta t^{-1}\sum_{i=1,2}\Big(\cfrac{1-\gamma}{2}+\cfrac{\gamma}{2}\Big)\|\delta\eta_{i}^{n+1}\|^{2}+\int_{\Omega}(\cfrac{\kappa_{x}\kappa_{u}(P_{H}u^{n})}{2})|\nabla\eta^{n+1}|^{2}+(\cfrac{1-C_{1}}{4})\|\delta\eta^{n+1}\|_{\kappa}^{2}\\
		\leq & \Delta t^{-1}\sum_{i=1,2}\cfrac{\gamma}{2}\|\delta\eta_{i}^{n}\|^{2}+\int_{\Omega}(\cfrac{\kappa_{x}\kappa_{u}(P_{H}u^{n})}{2})|\nabla\eta^{n}|^{2}-I_{1}+(r_{1}^{n},\delta\eta_{1}^{n+1})+(r_{2}^{n},\delta\eta_{2}^{n+1}).
	\end{align*}
	
	We next estimate the terms $I_{1}$ 
	\begin{align*}
		|I_{1}| & \leq C_{u'}\int_{\Omega}\kappa_{x}|\eta^{n+1}||\nabla u_{H}^{n}||\nabla\delta\eta^{n+1}|
		 \leq F_{1}\|\eta^{n+1}\|^{2}+(\cfrac{1-C_{1}}{8})\int_{\Omega}\kappa_{0}|\nabla\delta\eta^{n+1}|^{2}
	\end{align*}
	where we use the Lipschitz continuity of $\kappa_u$, and $F_{1}=\cfrac{2C_{0}^{2}C_{u'}^{2}}{(1-C_{1})\underline{\kappa}_{u}}$.
	Furthermore, 
	\begin{align*}
		 & \Big|\int_{\Omega}\kappa_{x}(\kappa_{u}(P_{H}u^{n+1})-\kappa_{u}(P_{H}u^{n}))|\nabla\eta^{n+1}|^{2}\Big|\\
		\leq & C_{u'}\int_{\Omega}\kappa_{x}|P_{H}(\delta u^{n+1})||\nabla\eta^{n+1}|^{2}
		\leq  \Delta t D_1\|(\kappa_{x}\kappa_{u}(P_{H}u^{n+1}))^{\frac{1}{2}}\nabla\eta^{n+1}\|^2_{L^{2}},
	\end{align*}
 where $D_1=C_0 C_{u'}$
	Therefore, we have
	\begin{align*}
		& \frac{1}{\Delta t}\sum_{i=1,2}\Big(\cfrac{1-\gamma}{2}+\cfrac{\gamma}{2}\Big)\|\delta\eta_{i}^{n+1}\|^{2}+\int_{\Omega}(\cfrac{\kappa_{x}\kappa_{u}(P_{H}u^{n+1})}{2})|\nabla\eta^{n+1}|^{2}+(\cfrac{1-C_{1}}{4})\|\delta\eta^{n+1}\|_{\kappa}^{2}\\
		\leq & \frac{1}{\Delta t}\sum_{i=1,2}\cfrac{\gamma}{2}\|\delta\eta_{i}^{n}\|^{2}+\int_{\Omega}(\cfrac{\kappa_{x}\kappa_{u}(P_{H}u^{n})}{2})|\nabla\eta^{n}|^{2}+(r_{1}^{n},\delta\eta_{1}^{n+1})+(r_{2}^{n},\delta\eta_{2}^{n+1})\\
		& +\Delta t D_1\|(\kappa_{x}\kappa_{u}(P_{H}u^{n+1}))^{\frac{1}{2}}\nabla\eta^{n+1}\|^2_{L^{2}}+F_{1}\|\eta^{n+1}\|^{2}+(\cfrac{1-C_{1}}{8})\|\delta\eta^{n+1}\|_{\kappa}^{2}.
	\end{align*}
	and thus,
	\begin{align*}
		& \sum_{i=1,2}\cfrac{1+\gamma}{4\Delta t}\|\delta\eta_{i}^{n+1}\|^{2}+(\cfrac{1-C_{1}}{8})\|\delta\eta^{n+1}\|_{\kappa}^{2}+\Big(1-D_1\Delta t\Big)\|\eta^{n+1}\|^2_{\kappa(u^{n+1})}\\
		\leq & \sum_{i=1,2}\cfrac{\gamma}{2\Delta t}\|\delta\eta_{i}^{n}\|^{2}+\|\eta^n\|^2_{\kappa(u^n)}+F_{1}\|\eta^{n+1}\|^{2}+(r_{1}^{n},\delta\eta_{1}^{n+1})+(r_{2}^{n},\delta\eta_{2}^{n+1})
	\end{align*}
 This gives the result in Lemma \ref{lem:error_a}. 

We will next discuss the proof of Lemma \ref{lem:error_L2}.
We consider $v=\eta_{1}^{n},\eta_{2}^{n}$ in (\ref{eq:residual_1}) and (\ref{eq:residual_2}),
	we obtain
	\begin{align*}
		(\cfrac{\delta\eta_{1}^{n+1}+\delta\eta_{2}^{n}}{\Delta t},\eta_{1}^{n})+(\cfrac{\delta\eta_{2}^{n+1}+\delta\eta_{1}^{n}}{\Delta t},\eta_{2}^{n})\\
		+a_{0}(\eta_{1}^{n+1}+\eta_{2}^{n},\eta^{n})+b(P_{H}u^{n},u_H^n,\eta^{n}) & =(r_{1}^{n},\eta_{1}^{n})+(r_{2}^{n},\eta_{2}^{n})
	\end{align*}
	
	First, we have 
	\begin{align*}
		& (\cfrac{\delta\eta_{1}^{n+1}+\delta\eta_{2}^{n}}{\Delta t},\eta_{1}^{n})+(\cfrac{\delta\eta_{2}^{n+1}+\delta\eta_{1}^{n}}{\Delta t},\eta_{2}^{n})
		= \sum_{i=1,2} (\cfrac{\delta\eta_{i}^{n+1}-\delta\eta_{i}^{n}}{\Delta t},\eta_{i}^{n})+(\cfrac{\delta\eta^{n}}{\Delta t},\eta^{n})\\
		= & \cfrac{1}{2\Delta t}\Big(\sum_{i=1,2}\Big((\|\eta_{i}^{n+1}\|^{2}-\|\eta_{i}^{n}\|^{2})-(\|\eta_{i}^{n}\|^{2}-\|\eta_{i}^{n-1}\|^{2})-(\|\delta\eta_{i}^{n+1}\|^{2}+\|\delta\eta_{i}^{n}\|^{2})\Big)\\
		& +\|\eta^{n}\|^{2}-\|\eta^{n-1}\|^{2}+\|\delta\eta^{n}\|^{2}\Big).
	\end{align*}
	
	We define $E_{n+1}(\eta):=\cfrac{1}{2\Delta t}\Big(\sum_{i=1,2}\Big((\|\eta_{i}^{n+1}\|^{2}-\|\eta_{i}^{n}\|^{2})+\|\eta^{n}\|^{2}\Big)$
	and simplify the equation as 
	\begin{align*}
		&(\cfrac{\delta\eta_{1}^{n+1}+\delta\eta_{2}^{n}}{\Delta t},\eta_{1}^{n})+(\cfrac{\delta\eta_{2}^{n+1}+\delta\eta_{1}^{n}}{\Delta t},\eta_{2}^{n})\\
		=&E_{n+1}(\eta)-E_{n}(\eta)+\cfrac{1}{2\Delta t}\Big(\|\delta\eta^{n}\|^{2}-\sum_{i=1,2}(\|\delta\eta_{i}^{n+1}\|^{2}+\|\delta\eta_{i}^{n}\|^{2})\Big).
	\end{align*}
	We next estimate the terms $a_{0}(\eta_{1}^{n+1}+\eta_{2}^{n},\eta^{n})+b(P_{H}u^{n},u_H^n,\eta^{n})$, which is 
	\begin{align*}
		&a_{0}(\eta_{1}^{n+1}+\eta_{2}^{n},\eta^{n})+\tilde{a}(P_{H}u^{n};P_{H}u^{n},\eta^{n})-\tilde{a}(u_{H}^{n};u_{H}^{n},\eta^{n})\\
		= & a_{0}(\eta_{1}^{n+1}+\eta_{2}^{n},\eta^{n})-a_{0}(\eta^{n},\eta^{n})+a(P_{H}u^{n};P_{H}u^{n},\eta^{n})-a(u_{H}^{n};u_{H}^{n},\eta^{n})\\
		= & a_{0}(\eta^{n+1},\eta^{n})-a_{0}(\delta\eta_{2}^{n+1},\eta^{n})-a_{0}(\eta^{n},\eta^{n})+a(P_{H}u^{n};P_{H}u^{n},\eta^{n})-a(u_{H}^{n};u_{H}^{n},\eta^{n}).
	\end{align*}
	where we use the property $a(P_{H}(v),w)=a(v,w),\;\forall w\in V_{H}$. For the first three terms in the above equation, we have
	\begin{align*}
		& a_{0}(\eta^{n+1},\eta^{n})-a_{0}(\delta\eta_{2}^{n+1},\eta^{n})-a_{0}(\eta^{n},\eta^{n})\\
		& =\cfrac{1}{2}\Big(\|\eta^{n+1}\|_{\kappa}^{2}-\|\eta^{n}\|_{\kappa}^{2}-\|\delta\eta^{n+1}\|_{\kappa}^{2}-2a_{0}(\delta\eta_{2}^{n+1},\eta^{n})\Big)\\
		& \geq\cfrac{1}{2}\Big(\|\eta^{n+1}\|_{\kappa}^{2}-(1+\cfrac{\underline{\kappa}_{u}}{2C_{u}})\|\eta^{n}\|_{\kappa}^{2}-\|\delta\eta^{n+1}\|_{\kappa}^{2}-C_{2}\|\delta\eta_{2}^{n+1}\|_{\kappa}^{2}\Big)
	\end{align*}
	where $C_{2}=\cfrac{2C_{u}}{\underline{\kappa}_{u}}$. And for the last two terms, we have 
	\begin{align*}
		& a(P_{H}u^{n};P_{H}u^{n},\eta^{n})-a(u_{H}^{n};u_{H}^{n},\eta^{n})\\
		= & \int_{\Omega}\kappa_{x}\Big(\kappa_{u}(P_{H}u^{n})-\kappa_{u}(u_{H}^{n})\Big)\nabla u_{H}^{n}\cdot\nabla\eta^{n}+\int_{\Omega}\kappa_{x}\kappa_{u}(P_{H}u^{n})|\nabla\eta^{n}|^{2}\\
		= & I_{2}+\int_{\Omega}\kappa_{x}\kappa_{u}(P_{H}u^{n})|\nabla\eta^{n}|^{2}
	\end{align*}
	where $I_{2}:=\int_{\Omega}\kappa_{x}\Big(\kappa_{u}(P_{H}u^{n})-\kappa_{u}(u_{H}^{n})\Big)\nabla u_{H}^{n}\cdot\nabla\eta^{n}$.
	We can show that 
	\begin{align*}
		|I_{2}| & \leq C_{u'}\int_{\Omega}\kappa_{x}|\eta^{n}||\nabla u_{H}^{n}||\nabla\eta^{n}|
		 \leq F_{2}\|\eta^{n}\|^{2}+\cfrac{\underline{\kappa}_{u}}{4C_{u}}\|\eta^{n}\|_{\kappa}^{2},
	\end{align*}
	where $F_{2}=\underline{\kappa}_{u}^{-2}C_{u}C_{0}^{2}C_{u'}^{2}.$
	Thus, we have 
	\begin{align*}
		& a_{0}(\eta_{1}^{n+1}+\eta_{2}^{n},\eta^{n})+\tilde{a}(P_{H}u^{n};P_{H}u^{n},\eta^{n})-\tilde{a}(u_{H}^{n};u_{H}^{n},\eta^{n})\\
		\geq & \cfrac{1}{2}\Big(\|\eta^{n+1}\|_{\kappa}^{2}-\cfrac{\underline{\kappa}_{u}}{4C_{u}}\|\eta^{n}\|_{\kappa}^{2}-\|\delta\eta^{n+1}\|_{\kappa}^{2}-C_{2}\|\delta\eta_{2}^{n+1}\|_{\kappa}^{2}\Big)\\&+\cfrac{1}{2}\int_{\Omega}\kappa_{x}\kappa_{u}(P_{H}u^{n})|\nabla\eta^{n}|^{2}+I_{2}\\
		\geq & \cfrac{1}{2}\Big(\|\eta^{n+1}\|_{\kappa}^{2}-\|\eta^{n}\|_{\kappa}^{2}-\|\delta\eta^{n+1}\|_{\kappa}^{2}-C_{2}\|\delta\eta_{2}^{n+1}\|_{\kappa}^{2}\Big)-F_{2}\|\eta^{n+1}\|^{2}\\
  &+\cfrac{1}{2}\int_{\Omega}\kappa_{x}\kappa_{u}(P_{H}u^{n})|\nabla\eta^{n}|^{2}
	\end{align*}
	
	Combining the equations, we obtain
	\begin{align*}
		& \Big(E_{n+1}(\eta)+\cfrac{1}{2}\|\eta^{n+1}\|_{\kappa}^{2}\Big)-\Big(E_{n}(\eta)+\cfrac{1}{2}\|\eta^{n}\|_{\kappa}^{2}\Big)+\cfrac{1}{2\Delta t}\|\delta\eta^{n}\|^{2} \\
  &+\cfrac{1}{2}\int_{\Omega}\kappa_{x}\kappa_{u}(P_{H}u^{n})|\nabla\eta^{n}|^{2}\\
		\leq & F_{2}\|\eta^{n+1}\|^{2}+\cfrac{1}{2}\Big(\|\delta\eta^{n+1}\|_{\kappa}^{2}+C_{2}\|\delta\eta_{2}^{n+1}\|_{\kappa}^{2}\Big)+\cfrac{1}{2\Delta t}\sum_{i=1,2}\Big(\|\delta\eta_{i}^{n+1}\|^{2}+\|\delta\eta_{i}^{n}\|^{2}\Big)\\
  &+(r_{1}^{n},\eta_{1}^{n})+(r_{2}^{n},\eta_{2}^{n}).
	\end{align*}
 This gives the result in Lemma \ref{lem:error_L2}. 
\end{document}